\documentclass[10pt,reqno]{amsart}
\usepackage[widespace]{fourier}
\usepackage{frcursive}
\usepackage{xcolor}
\usepackage{hyperref}
\hypersetup{colorlinks=true,linkcolor=blue, citecolor=blue}
\usepackage[matrix,arrow]{xy}
\usepackage{amsmath,amsthm,amscd,amsfonts,amssymb,euscript,latexsym}
\usepackage{mathtools}
\usepackage{pdflscape}

\input{xy}

\xyoption{all}

\newtheorem{ack}{Acknowledgments\!\!}

\numberwithin{equation}{subsection} 
\newtheorem{guess}{Theorem}[section]

\newtheorem{rem}[guess]{Remark}
\newtheorem{defi}[guess]{Definition}

\newtheorem{thm}[guess]{Theorem}

\newtheorem{lem}[guess]{Lemma}

\newtheorem{prop}[guess]{Proposition}

\newtheorem{Cor}[guess]{Corollary}

\newcommand{\gfr}{\mathfrak{g}}

\newcommand{\cU}{\mathcal{U}}

\newcommand{\cE}{\mathcal{E}}
\newcommand{\cR}{\mathcal{R}}
\newcommand{\cH}{\mathcal{H}}

\newcommand{\cA}{\mathcal{A}}

\newcommand{\cG}{\mathcal{G}}

\newcommand{\cP}{\mathcal{P}}

\newcommand{\hra}{\hookrightarrow}

\newcommand{\ra}{\rightarrow}
\newcommand{\ol}{\overline}

\newcommand{\RR}{\mathbb{R}}

\newcommand{\ZZ}{\mathbb{Z}}
\newcommand{\GG}{\mathbb{G}}
\newcommand{\QQ}{\mathbb{Q}}

\newcommand{\CC}{\mathbb{C}}

\newcommand{\spec}{{\rm Spec}\,}

\newcommand{\beqa}{\begin{eqnarray}}
\newcommand{\eeqa}{\end{eqnarray}}

\begin{document}

\title{On a ``Wonderful" Bruhat-Tits group scheme}

\author[Vikraman Balaji]{Vikraman Balaji}
\address{Chennai Mathematical Institute, Plot number H1, Sipcot IT Park, Siruseri, Chennai, India}
\email{balaji@cmi.ac.in}

\author[Y. Pandey]{Yashonidhi Pandey}
\thanks{The support of Science and Engineering Research Board under Mathematical Research Impact Centric Support File number: MTR/2017/000229 is gratefully acknowledged.}
\address{ 
Indian Institute of Science Education and Research, Mohali Knowledge city, Sector 81, SAS Nagar, Manauli PO 140306, India}
\email{ ypandey@iisermohali.ac.in, yashonidhipandey@yahoo.co.uk}

\begin{abstract} In this paper we make a universal construction of Bruhat-Tits group scheme on wonderful embeddings of adjoint groups in the absolute and relative settings. We make a similar construction for the wonderful embeddings of adjoint Kac-Moody groups. These have natural classifying properties reflecting the orbit structure on the wonderful embeddings. 

\end{abstract}

\subjclass[2000]{14L15,14M27,14D20}
\keywords{Bruhat-Tits group scheme, parahoric group, Loop groups, Wonderful compactification}
\maketitle
\small
\tableofcontents
\normalsize
%\Large

\input{amssym.def}
%\newtheorem{guess}[section]{Question}
%\input{•} macros
% Calligraphische Zeichen

\section{Introduction}
Let $G$ be an almost simple, simply-connected  group over an algebraically closed field $k$ of characteristic zero and let $G_{_{_{\text{ad}}}} := G/Z(G)$. The aim of this note is to construct  certain {\sl universal group schemes} on 
\begin{itemize}
\item the De Concini-Procesi wonderful compactification \cite{decp} $\bf X$ of $G_{_{_{\text{ad}}}}$, 
\item the loop ``wonderful embedding" ${\bf X}^{^{\text{aff}}}$ of the adjoint affine Kac-Moody group $G^{^{\text{aff}}}_{_{\text{ad}}}$, constructed by Solis \cite{solis},
\item certain toroidal embeddings $\bar{G}_{_{_{\text{ad,A}}}}$ of the relative group scheme $G_{_{_{\text{ad,A}}}}$ modeled after \cite{kkms}.
\end{itemize} 

The group schemes we construct are sufficiently universal to be called ``wonderful". A new point of view, playing a central role in this work, is that {\sl parabolic vector bundle} on {\sl logarithmic schemes} can be used as a {\it tool} to make geometric constructions. Hitherto, they have been  {\it objects} of study occuring as points in certain moduli spaces.

Let us briefly motivate the theme of this note.

\subsection{Group embeddings and buildings} 
Let us recall that Tits' buildings are basically of two types. The first one is the ``absolute" Tits building or spherical building which is attached to a semi-simple group over a general field. This simplicial complex is built from simplices which correspond to parabolic subgroups. The apartments of the building correspond to parabolic subgroups containing a fixed maximal torus. This is built up out of Euclidean spaces decomposed by the usual Weyl chambers. The second one is the Bruhat-Tits building which is the ``relative" building attached to a semi-simple group over a complete-valued field. This is based on its {\em parahoric subgroups}  and built up out of Euclidean spaces decomposed into affine Weyl chambers.

The two types of buildings can also be seen from an algebro-geometric perspective. In the absolute case, we work with a semisimple group $G_{_{\text{ad}}}$   of adjoint type. In this setting one has the  wonderful embedding $G_{_{\text{ad}}} \subset {\bf X}$
where $G_{_{\text{ad}}}$ sits as an open dense subset of $\bf X$.
The complement ${\bf X} \setminus G_{_{\text{ad}}}$ is stratified by subsets $Y$ and there is a bijection $Y \mapsto \{P_{_Y}| B \subset P_{_Y} \}$ from these strata to parabolic subgroups $P_{_Y} \subset G$ containing a fixed Borel subgroup $B$. Furthermore, this bijection extends to an isomorphism between the Tits building and the canonical complex associated with the toroidal embedding $G_{_{\text{ad}}} \subset {\bf X}$ (see Mumford \cite[Page 178]{kkms}). 

A second perspective is when  the ground field is endowed with a complete non-archimedean discrete valuation. Let $A = k\llbracket z \rrbracket$ be a complete discrete valuation ring, $K = k(\!(z)\!)$ its quotient field. In this setting our basic model was constructed by Mumford \cite{kkms}. He constructs  a toroidal embedding $G_{_{\text{ad, A}}} \subset \bar{G}_{_{\text{ad, A}}}$ of the split group scheme $G_{_{\text{ad, A}}} = G_{_{\text{ad, A}}} \times \spec A$. The strata of $\bar{G}_{_{\text{ad, A}}}\setminus  G_{_{\text{ad, A}}}$ correspond bijectively to parahoric subgroups of $G(K)$ in a way that naturally extends to an isomorphism of the graph of the embedding $G_{_{\text{ad, A}}} \subset \bar{G}_{_{\text{ad, A}}}$ with the Bruhat-Tits building of $G \times \spec A$ over $A$. 

%In \cite{remyetal}, the authors study the building and compactifications in the rigid analytic setting. They  construct a wonderful compactification $\bar{G}^{^{an}}$ of the building in the setting of Berkovich spaces.

\subsection{Statement of main results} Classical Bruhat-Tits theory associates, to each facet $\Sigma$ of the Bruhat-Tits building, a smooth group scheme $\cG_{_{\Sigma}}$ on $\spec~A$, with connected fibres and whose generic fibre is $G \times _{_{\spec k}} \spec~K$. We call $\cG_{_{\Sigma}}$ a Bruhat-Tits group scheme on $\spec A$. The $A$-valued points $\cG_{_{\Sigma}}(A) \subset G(K)$ are precisely the parahoric subgroups of $G(K)$. In this paper we construct universal analogues of the Bruhat-Tits group scheme.

\subsubsection{The case of Tits building}  In the first setting, namely in the case of the Tits building, we construct an affine group scheme $\cG_{_{\bf X}}$ over $\bf X$ whose restriction along each curve transversal to a strata of ${\bf X} \setminus G_{_{\text{ad}}}$ corresponds to the Bruhat-Tits group scheme associated to the parabolic subgroup under the bijection $Y \mapsto \{P_{_Y}| B \subset P_{_Y} \}$ mentioned above. 

To state our theorem, we introduce some relevant notations and notions. Let ${\bf X} := \ol{G_{_{_{\text{ad}}}}}$ be the wonderful compactification of $G_{_{_{\text{ad}}}}$. We construct a locally free sheaf of Lie algebras on $\bf X$. This construction is essentially toric in the sense that it is first constructed on the toric varieties based on the negative Weyl chambers and then the ones on the bigger spaces is deduced from this construction.  We endow the locally free sheaf $\mathcal R$ with a canonical parabolic structure at the generic points of the divisor ${\bf X} \setminus G_{_{_{\text{ad}}}}$ together with a compatible loop Lie algebra structure. 

We fix data $(T,B,G)$ of $G$. Let $S$ denote the set of simple roots of  $G$ and $\mathbb S=S \cup \{\alpha_0\}$ denote the set of affine simple roots. For $\emptyset \neq \mathbb I \subset \mathbb S$ let $\mathcal{G}_{\mathbb I}$ denote the associated Bruhat-Tits group scheme on a dvr and for any $I \subset S$ let $\mathcal{G}^{^{st}}_{I}=\mathcal{G}_{\mathbb I}$ where $\mathbb I =I \cup \{\alpha_0\}$.
For ${I} \subset S$, let $Z_{I}$ denote the corresponding strata of $\bf X$ and for any point $z_{_{I}} \in Z_{_{I}}$, let $C_{_{I}} \subset {\bf X}$ be a smooth curve with generic point in $G_{_{\text{ad}}}$ and closed point $z_{_{I}}$. Let $U_{_{z}} \subset C_{_{I}}$ be a formal neighbourhood of  $z_{_{I}}$.

We these notations,  Theorem \ref{btoverx} is as follows.
\begin{thm}\label{btoverx00}  There exists an affine ``wonderful" Bruhat-Tits group scheme ${\mathcal G}_{_{\bf X}}^{^{\varpi}}$ on $\bf X$ satisfying the following classifying properties. 
\begin{enumerate} 
\item There is an identification of the Lie-algebra bundles $\text{Lie}({\mathcal G}_{_{\bf X}}^{^{\varpi}}) \simeq \mathcal R$.

\item For $\emptyset \neq I \subset S$ the restriction of ${\mathcal G}_{_{{\bf X}}}^{^{\varpi}}$ to  the formal neighbourhood $U_{_{z_{_I}}}$ of $z_{_I}$ in $C_{_I}$ \S \eqref{constructionR} is isomorphic to the standard Bruhat-Tits group scheme $\cG^{^{st}}_{_{I}}$ \S \eqref{sbtgpsch}.
\end{enumerate} 
\end{thm}
In \eqref{pabsolute}, we have indicated generalizations to fields of positive characteristics.

\subsubsection{The relative case of the Bruhat-Tits buiding} In the second scenario we  work in the setting of loop groups and construct an affine group scheme over a ``wonderful" embedding ${\bf X}^{^{\text{aff}}}$ constructed by Solis \cite{solis}. In the relative case the group scheme is obtained by ``integrating" a locally free sheaf of Lie algebras $\bf R$ on ${\bf X}^{^{\text{aff}}}$.  Its construction is achieved by constructing a locally free sheaf of Lie-algebras $J$ on a finite dimensional scheme ${\bf Y}^{^{\text{aff}}}$ which is the closure of a torus-embedding and whose translates build up the ind-scheme ${\bf X}^{^{\text{aff}}}$. The sheaf $J$ comes equipped with a canonical parabolic structure on a normal crossing divisor. The sheaf $J$  plays the role analogous to $\mathcal{R}$ (on $\bf X$) once we view ${\bf Y}^{^{\text{aff}}}$ as built out of affine Weyl group translates of the affine space $\mathbb{A}^{\ell+1}$ whose standard coordinate hyperplanes play the role of strata of largest dimension of $\bf X$. Using this perspective, we then briefly consider a toroidal embedding $\bar{G}_{_{ad, A}}$ the structure of which is modeled after Mumford's construction in \cite{kkms}. We then define an affine group scheme over $\bar{G}_{_{ad, A}}$ which has properties analogous to those of $\cG_{_{\bf X}}$.  

We first introduce some notation to state our theorem more precisely. Let $LG$ denote the loop group of $G$. Let $L^{\ltimes}G= \GG_m \ltimes LG$  where the {\it rotational torus} $\GG_m$ acts on $LG$ by acting on the uniformizer (cf. \S \ref{loopcase}). Let $G^{^{\text{aff}}}$ denote the Kac-Moody group associated to the affine Dynkin diagram of $G$. Recall that $G^{^{\text{aff}}}$ is given by a  central extension of $L^{\ltimes}G$ by $\GG_m$. 

Let $T_{_{\text{ad}}} := T/Z(G)$ and let us denote by $T^{\ltimes}_{_{\text{ad}}}$ the torus $ \GG_m \times T_{_{\text{ad}}} \subset G^{^{\text{aff}}}_{_{\text{ad}}}$. In ${\bf X}^{^{\text{aff}}}$, the closure ${\bf Y}^{^{\text{aff}}}:=\ol{T^{\ltimes}_{_{\text{ad}}}}$ gives  a  torus-embedding. The complement $Z:= {\bf Y}^{^{\text{aff}}} \setminus T^{\ltimes}_{_{\text{ad}}}$ is a union $\cup_{_{\alpha \in \mathbb S}} H_{_\alpha}$ of $\ell + 1$ smooth divisors meeting at normal crossings.  For $\alpha \in \mathbb S$, let $\xi_{_\alpha} \in H_{_\alpha}$ denote the generic points of the divisors $H_{_\alpha}$'s. Let  $A_{_\alpha}= \mathcal O_{_{{\bf Y}^{^{\text{aff}}}, \xi_{_{\alpha}}}}$
be the dvr's obtained by localizing at the height $1$-primes given by the $\xi_{_\alpha}$'s.  Let $Y_{_\alpha}:= \spec(A_{_\alpha}) $. In Theorem \ref{liealgbunonyaff} we show the existence of a finite dimensional Lie-algebra bundle $J$ on ${\bf Y}^{^{\text{aff}}}$ which extends the trivial bundle with fiber $ \mathfrak g$ on the open dense subset ${T^{\ltimes}_{_{\text{ad}}}}  \cap {\bf Y}^{^{\text{aff}}} \subset {\bf Y}^{^{\text{aff}}}$ and for each $\alpha \in \mathbb S$ we have $L^+(J_{_{Y_{_\alpha}}}) \simeq  L^+({\text{Lie}}(\mathcal{G}_{_\alpha}))$. Then in Proposition \ref{isotliealginfty1} we show the existence of a finite dimensional Lie-algebra bundle $\bf R$ on ${\bf X}^{^{\text{aff}}}$ which extends the trivial Lie algebra bundle $G_{_{\text{ad}}}^{^{\text{aff}}} \times \mathfrak g$ on the open dense subset $G_{_{\text{ad}}}^{^{\text{aff}}} \subset {\bf X}^{^{\text{aff}}}$ and whose restriction to ${\bf Y}^{^{\text{aff}}}$ is $J$. 

The ind-scheme ${\bf X}^{^{\text{aff}}}$ has divisors $D_\alpha$ for $\alpha \in \mathbb S$ such that the complement of their union is ${\bf X}^{^{\text{aff}}} \setminus G_{ad}^{aff}$. 
With these notations, let us state  Theorem \ref{gpshsolis}.
\begin{thm} There exists an affine ``wonderful" Bruhat-Tits group scheme ${\mathcal G}_{_{{\bf X}^{^{\text{aff}}}}}^{^{\varpi}}$ on ${\bf X}^{^{\text{aff}}}$ together with a canonical isomorphism $\text{Lie}({\mathcal G}_{_{\bf X^{aff}}}^{^{\varpi}}) \simeq {\bf R}$. It further satisfies the following classifying property: 

For $h \in {\bf X}^{^{\text{aff}}} \setminus G_{ad}^{aff}$, let 
$ \mathbb I \subset \mathbb S$ be the subset such that $h \in \cap_{\alpha \in I} D_\alpha$. Let $C_{_{\mathbb I}} \subset {\bf X}^{^{\text{aff}}}$ be a smooth curve with generic point in $G^{aff}_{_{\text{ad}}}$ and closed point $h$. Let $U_{_{h}} \subset C_{_{\mathbb I}}$ be a formal neighbourhood of the closed point $h$. Then, the restriction ${\mathcal G}_{_{{\bf X}^{^{\text{aff}}}}}^{^{\varpi}}|_{_{U_{_{h}}}}$ is isomorphic to the Bruhat-Tits group scheme $\cG_{_{\mathbb I}}$ \S \eqref{btgpsch} on $U_{_{h}}$. 
\end{thm}

%The layout of the paper is as follows. In \S3 we show the existence of the group scheme on $\bf X$, in \S4 on the torus-embedding ${\bf Y}^{^{\text{aff}}}$, in \S5 on ${\bf X}^{^{\text{aff}}}$ and in \S6 on ${\bf X}^{^{poly}}$. In \S7 we work with an analogue of Mumford's construction and construct the group scheme on it. The last \S8 is an appendix which collects various facts from the theory of parabolic bundles and related stuff for assisting the reader.
\begin{ack}{We thank J. Heinloth and M.Brion for their questions and comments. They  have helped us very much in expressing our results with greater precision. 
}\end{ack}
\section{Preliminaries}
\subsubsection{Lie-data of $G/k$} Let $G$ be an almost simple, simply-connected  group over $k$ (see \eqref{pabsolute}) with the data $(T,B,G)$. Let $X(T)\,=\,\text{Hom}(T, \GG_{_m})$ be the group of characters of $T$ and  $Y(T)\,=\, \text{Hom}(\GG_m, T)$ be the group of all one--parameter subgroups of $T$. 
Let $G_{_{_{\text{ad}}}} := G/Z(G)$ and let $\gfr$ denote the Lie-algebra. We denote $\Phi^{+},\Phi^{-} \subset \Phi$  the set of positive and negative roots with respect to $B$. Let $S= \{\alpha_{_{1}}, \ldots, \alpha_{_{\ell}}\}$ denote the set of simple roots of $G$, where $\ell$ is the rank of $G$.  Let $\alpha^{\vee}$ denote the  coroot  corresponding to $\alpha \in S$ . 

\subsubsection{Apartment data} Let  $\mathbb S = S \cup \{ \alpha_0 \}$ denote the set of affine simple roots. Let $\cA_T$ denote the affine apartment corresponding to $T$. It can be identified with the  affine space ${\mathbb E}=Y(T) \otimes_{\ZZ} \RR$ together with its origin $0$. 

Let $\mathbf{a}_0$ be the unique Weyl alcove of $G$ whose closure contains $0$ and which is contained in the dominant Weyl chamber corresponding to $B$.    Under the natural pairing between $Y(T) \otimes_{\ZZ} \QQ$ and $X(T) \otimes_\ZZ \QQ$, the integral basis elements dual to $S$ are called the fundamental co-weights $\{\omega^{\vee}_{\alpha}| \alpha \in S\}$. Let $c_{\alpha}$ be the coefficient of $\alpha$ in the highest root. The vertices of the Weyl alcove $\mathbf{a}_0$ are indexed by $0$ and 
\beqa\label{alcovevertices}
\theta_{_{\alpha}} := {{\omega^{\vee}_{\alpha}} \over {c_{\alpha}}}, \alpha \in S.
\eeqa
Indeed, any rational point $\theta \in \mathbf{a}_0$ can be expressed as $\theta_{_\lambda}$ so that there is a unique pair  $(d, \lambda) \in \mathbb{N} \times Y(T)$ defined by the condition that $d$ is the least positive integer such that 
\begin{equation} \label{dtheta}
\lambda = d. \theta_{_\lambda} \in Y(T).
\end{equation}
Thus, if $e_\alpha$ is the order of $\omega^\vee_\alpha$ in the quotient of the co-weight lattice by $Y(T)$, it follows that for $\theta_{_\alpha}$, the number $d$ is: 
\begin{equation} \label{dalpha} d_{_\alpha} := e_\alpha .c_{\alpha}.
\end{equation}

An affine simple root $\alpha \in \mathbb S$ may be viewed as an affine functional on $\cA_T$. Any non-empty subset $\mathbb I \subset \mathbb S$ defines the facet $\Sigma_\mathbb I \subset \ol{\mathbf{a}_0}$ where exactly the $\alpha$'s not lying in  $\mathbb I$ vanish. So $\mathbb S$ corresponds to the interior of the alcove and the vertices of the alcove correspond to $\alpha \in \mathbb S$. Conversely any facet $\Sigma \subset \ol{\mathbf{a}_0}$ defines non-empty subset $\mathbb I \subset \mathbb S$. For $\emptyset \neq \mathbb I \subset \mathbb S$, the barycenter of $\Sigma_\mathbb I$ is given by  
 \begin{equation} \label{Ibarycenter}
 \theta_\mathbb I:= \frac{1}{|\mathbb I|} \sum_{\alpha \in \mathbb I} \theta_{_\alpha}.
 \end{equation}
 
\subsubsection{Loop groups and their parahoric subgroups}  \label{loopgpsec}
Let $k=\mathbb C$, $A=k\llbracket z \rrbracket$ and let $R$ denote an arbitrary $k$-algebra. We denote the field of Laurent polynomials by $K=k(\!(z)\!)= k\llbracket z \rrbracket[z^{-1}]$. The loop group $LG$ on a $k$-algebra $R$ is given by $G(R(\!(t)\!))$. Similarly the loop Lie algebra $L \gfr$ is given by $L \gfr(R)=\gfr(R(\!(t)\!))$. We can similarly define the positive loops (also called {\it jet groups} in the literature) 
$L^+(G)$ to be the subfunctor of $LG$ defined by $L^+(G)(R) := G(R\llbracket z \rrbracket)$. The positive loop construction extends more generally to any group scheme $\mathcal{G} \rightarrow \spec(A)$ and any vector bundle $\mathfrak{P} \rightarrow \spec(A)$ whose sheaf of sections carries a Lie-bracket.

Let $L^{\ltimes}G= \GG_m \ltimes LG$  where the {\it rotational torus} $\GG_m$ acts on $LG$ by acting on the uniformizer via the {\it loop  rotation action} as follows: $u \in \GG_m(R)$ acts on $\gamma(z) \in LG(R)=G(R(\!(t)\!))$ by $u\gamma(z)u^{-1}=\gamma(uz)$. A maximal torus of $L^{\ltimes}G$ is $T^{\ltimes}=\GG_m \times T$.  A $\eta \in Hom(\GG_m,T^{\ltimes}) \otimes_{\ZZ} \QQ$ over $R(s)$ can be viewed as a rational $1$-PS, i.e., a $1$-PS $\GG_m \ra T^{\ltimes}$ over $R(w)$ where $w^n=s$ for some $n \geq 1$. In this case for $\gamma(z) \in L^{\ltimes}G(R)$ we will view $\eta(s) \gamma(z) \eta(s)^{-1}$ as an element in $L^{\ltimes} G(R(w))$. Hence, by the condition on $\gamma(z) \in L^{\ltimes}G(R)$  {\begin{center} {``$\lim_{_{s \ra 0}} \eta(s) \gamma(z) \eta(s)^{-1}$ exists in $L^{\ltimes}G(R)$",}\end{center}} 
\noindent
we mean that  there exists a $n \geq 1$ such that for $w^n=s$, we have \begin{equation} \eta(s) \gamma(z) \eta(s)^{-1} \in L^{\ltimes}G(R \llbracket w \rrbracket ).
\end{equation}
If $\eta=(a,\theta)$ for $a \in \mathbb{Q}$ and $\theta$ a rational $1$-PS of $T$, then we have
\begin{equation} \label{loopconjugation}
\eta(s) \gamma(z) \eta(s)^{-1}=\theta(s) \gamma(s^{a}z) \theta(s)^{-1}.
\end{equation}
We note further that for any $0<d \in \mathbb{N}$, setting $\eta=(\frac{a}{d}, \frac{\theta}{d})$ we have \begin{equation}
\eta(s) \gamma(z) \eta(s)^{-1}=\theta(s^{\frac{1}{d}}) \gamma(s^{\frac{a}{d}}z) \theta(s^{\frac{1}{d}})^{-1}.
\end{equation}
So for $\eta=\frac{1}{d}(1,\theta)$ observe that the statement "$\lim_{s \rightarrow 0} \eta(s) \gamma(z) \eta(s)^{-1}$ exists " is a condition which is equivalent to  
\begin{equation}\label{observation} \theta(s) \gamma(s) \theta(s)^{-1} \in L^{\ltimes}G(R \llbracket w \rrbracket ).
\end{equation} 
In other words, the condition of existence of limits is independent of $d$,  and we may further set $z=s$ in $\gamma(z)$. More generally, if $a>0$ the condition for $(a,\theta)$ and $(1,\frac{\theta}{a})$ are equivalent. We may write this condition on $\gamma(s) \in L^{\ltimes}G(R)$ or $ LG(R)$ as 
\begin{equation} \label{meaning} "\lim_{s \rightarrow 0} \theta(s) \gamma(s) \theta(s)^{-1} \quad \text{ exists in} \quad L^{\ltimes}G(R) \quad \text{or} \quad LG(R) \quad \text{if} \quad \gamma(z) \in LG(R)"
\end{equation}
For $r \in \Phi$, let $u_r: \mathbb{G}_a \rightarrow G$ denote the root subgroup. If for some $b \in \mathbb{Z}$ and $t(z) \in R\llbracket z \rrbracket$ we take $\gamma(z):=u_r(z^{b} t(z)) \in Lu_r(R)$ , and $\eta:=(1,\theta)$, then 
\begin{equation}
\eta(s) u_r(z^{b} t(z)) \eta(s)^{-1}=\theta(s) u_r((sz)^{b} t(sz)) \theta(s)^{-1}=u_r( s^{r(\theta)} (sz)^{b} t(sz)). 
\end{equation}
So for $\eta=(1,\theta)$ the condition that the limit exists is equivalent to 
\begin{equation} \label{floorcondition} r(\theta)+b \geq 0 \iff \lfloor r(\theta)+ b \rfloor = \lfloor r(\theta) \rfloor + b \geq 0 \iff b \geq -\lfloor r(\theta) \rfloor.
\end{equation}
We note the independence of the above implications on the number $d$ occurring in the equation $\eta:=\frac{1}{d}(1,\theta)$.

Let $\pi_{_1}: T^{\ltimes} \ra \GG_m$ be the first projection. For any rational $1$-PS $\eta: \GG_m \ra T^{\ltimes}$ we say $\eta$ is positive if $\pi_{_1} \circ \eta >0$, negative if $\pi_{_1} \circ \eta<0$ and non-zero if it is either positive or negative. 

Any non-zero $\eta=(a,\theta)$ defines the following positive-loop functors from the category of $k$-algebras to the category of groups and Lie-algebras:
 \begin{eqnarray} \label{parlimstozero}
 {\cP}_{_\eta}^{\ltimes}(R):= \{ \gamma \in L^{\ltimes}G(R) | \quad \lim_{s \ra 0} \eta(s) \gamma(z) \eta(s)^{-1} \quad \text{exists in} \quad L^{\ltimes}G(R) \}, \\
 \label{liealglimstozero}
 {\mathfrak{P}}_{_\eta}^{\ltimes}(R):= \{ h \in L^{\ltimes} \gfr(R) | \quad \lim_{s \ra 0} {Ad}(\eta(s))(h(z)) \quad \text{exists in} \quad L^{\ltimes} \gfr(R) \}, \\
 \label{liealglimstozero1}
 {\cP}_{_\eta}(R):= \{ \gamma \in LG(R) | \quad \lim_{s \ra 0} \eta(s) \gamma(z) \eta(s)^{-1} \quad \text{exists in} \quad LG(R) \}, \\
 \label{paraLiealg}
 {\mathfrak{P}}_{_\eta}(R)=\{ h \in L \gfr(R) | \quad \lim_{s \ra 0} Ad(\eta(s))(h(z)) \quad \text{exists in} \quad L \gfr(R)\},\\
 \label{liealgpara}
\text{Thus},~~ \cP_{_\eta}:=\cP_{_\eta}^{\ltimes} \cap (1 \times LG) \quad \text{and} \quad  {\mathfrak{P}}_{_\eta}:={\mathfrak{P}}_{_\eta}^{\ltimes} \cap  (0 \oplus L \gfr).
 \end{eqnarray}

A {\it parahoric} subgroup  of $L^{\ltimes}G$ (resp. $LG$) is a subgroup that is conjugate to $\cP^{\ltimes}_{_\eta}$ (resp. $\cP_{_\eta}$) for some $\eta$. In this paper, we will mostly be using only the case when $\eta=\frac{1}{d}(1,\theta)$. In this case, letting $L\mathfrak{g}(R)=\mathfrak{g}(R(\!(s)\!))$ as in \eqref{meaning}, we may reformulate \eqref{liealgpara} as 
\begin{equation} \label{meaning1} 
{\mathfrak{P}}_{_\eta}(R)=\{h \in L\mathfrak{g}(R) | \quad \lim_{s \rightarrow 0} Ad(\theta(s))(h(s)) \quad \text{exists in} \quad L\mathfrak{g}(R) \}.
\end{equation}
By the conditions (\ref{parlimstozero}) and (\ref{liealgpara}), using (\ref{floorcondition}), we may express $\cP_{_\eta}$ in terms of generators  as follows:
\begin{equation}
\cP_{_\eta}(R)=<T(R\llbracket z \rrbracket), u_r( z^{-\lfloor (r,\theta) \rfloor} R\llbracket z \rrbracket), r \in \Phi>,
\end{equation}
Let $\mathfrak{t}$ denote the Cartan sub-algebra of $T$ and $\mathfrak{u}_r$ denote the root algebras associated to $u_r$. Then the Lie-algebra functor of $\cP_{_\eta}$ in terms of generators is given by 
\begin{equation}
Lie(\cP_{_\eta})(R)=<\mathfrak{t}(R\llbracket z \rrbracket), \mathfrak{u}_r(z^{-\lfloor (r,\theta) \rfloor} R\llbracket z \rrbracket), r \in \Phi>.
\end{equation}
Thus, for any $\eta$ of the type $(1,\theta)$,  we get the equality
\begin{equation} \label{parahorLiealgrel}
\text{Lie}(\cP_{_\eta})={\mathfrak{P}}_{_\eta}.
\end{equation}
This can be seen by  (\ref{liealglimstozero}) and (\ref{liealgpara}) and then  replacing the conjugation in (\ref{loopconjugation})  by $Ad(\eta(s))$.

For a rational $1$-PS $\theta$ of $T$, with $\eta=(1,\theta)$ we will have the notation: 
\begin{equation} \label{paratheta}
{\mathfrak{P}}_{_\theta}:={\mathfrak{P}}_{_\eta}.
\end{equation}
In particular, for $r \in \Phi$ let $\mathfrak g_{_r} \subset \mathfrak g$ be the root subspace. Let $m_{_r}(\theta) := -\lfloor{(r,\theta)}\rfloor$. The parahoric Lie algebra ${\mathfrak P}_{_{\theta}}(k) \subset \mathfrak g(K)$ has $T$-weight space decomposition as:
\beqa \label{paracartan}
{\mathfrak P}_{_{\theta}}(k) = \mathfrak t(A) \bigoplus z^{^{m_{_r}(\theta)}} \mathfrak g_{_r}(A).
\eeqa

\subsubsection{Bruhat-Tits group scheme} \label{btgpsch}
To each facet $\Sigma_\mathbb I \subset \ol{\mathbf{a}_0}$, Bruhat-Tits theory associates  a smooth group scheme $\cG_{_{\mathbb I}}$ on $\spec~A$, with connected fibres and whose generic fibre is $G \times _{_{\spec k}} \spec~K$. We call $ \cG_{_{\mathbb I}}$ a Bruhat-Tits group scheme on $\spec A$. To $\cG_{_{\mathbb I}}$, we can associate a pro-algebraic group $L^+(\cG_{_{\mathbb I}})$ over $k$ as follows:
\begin{equation}
L^+(\cG_{_{\mathbb I}})(R) :=  \cG_{_{\mathbb I}}(R\llbracket z \rrbracket).
\end{equation}
The $L^+(\cG_{_{\mathbb I}})$ also characterise $\cG_{_{\mathbb I}}$. We thus have the identifications: 
\begin{equation} L^+(\cG_{_{\mathbb I}}) = \cP_{_\eta}  \quad \text{and} \quad \text{Lie}(L^+(\cG_{_{\mathbb I}})) = {\mathfrak{P}}_{_\eta}
\end{equation}
for $\eta=(1,\theta)$ where $\theta$ is any rational $1$-PS lying in $\Sigma_\mathbb I$.

\subsubsection{Standard Parahoric subgroups} \label{sbtgpsch}
The {\em standard parahoric subgroups} of $G(K)$ are parahoric subgroups of the distinguished hyperspecial parahoric subgroup  $G(A)$. These are realized as inverse images under the evaluation map $ev: G(A) \to G(k)$ of standard parabolic subgroups of $G$. In particular, the standard {\em Iwahori subgroup} ${\mathfrak I}$ is a standard parahoric  and indeed, ${\mathfrak I} = ev^{-1}(B)$. Denoting by $Q_{_I} \subset G$ the parabolic subgroup associated to  the subset $I \subset S$ we will denote by $\cG^{^{st}}_{_{{I}}}$
the {\em standard parahoric subgroups} of $G(A)$ determined by 
\begin{equation} \label{standard} L^+(\cG^{^{st}}_{_{{I}}})(k):= ev^{-1}(Q_{_{I}}).
\end{equation} 
Thus, for ${I} \subset S$, setting $\mathbb I := I \cup \{ \alpha_0  \}$ we have 
\begin{equation} \label{relnstdnonstd} L^+(\cG^{^{st}}_{_{{I}}})(k) = L^+(\cG_{_{\mathbb I}})(k) \quad \text{and} \quad L^+(\mathfrak P_{_{I}}^{^{st}})=L^+(\mathfrak P_{_{\mathbb{I}}}).
\end{equation}

%For $\alpha \in S$, when ${I}_{_\alpha} := \{\alpha\}$, then the subgroup $Q_{_{{I}_{_\alpha}}} \subset G$ is a maximal parabolic subgroup. Thus $ev^{-1}(Q_{_{{I}_{_\alpha}}}) = L^+(\cG^{^{st}}_{_{{I}_{_\alpha}}})(k)$ is a standard parahoric subgroup while $L^+(\cG_{_{{\mathbb I}_{_\alpha}}})= {\mathcal P}_{_{\theta_{\alpha}}}$. Further we have:  \begin{eqnarray} \label{stdnstdgp} L^+(\cG^{^{st}}_{_{{I}_{_\alpha}}})(k) = L^+(\cG_{_{{\mathbb I}_{_\alpha}}})(k) \cap G(A), \\ \label{stdnstdlie} {\mathfrak P}^{st}_{_{{I}_{_\alpha}}}(k) =  {\mathfrak P}_{_{\theta_{\alpha}}}(k) \cap \gfr(A). \end{eqnarray} To summarize: \begin{notat}\label{stbt} For ${I} \subset S$ and $\emptyset \neq \mathbb I \subset \mathbb S$, let $\cG^{^{st}}_{_{{I}}}$ (resp.  $\cG_{_{\mathbb I}}$) denote the Bruhat-Tits group scheme on $\spec(A)$ associated to the parahoric subgroup  $L^+(\cG^{^{st}}_{_{{I}}})$ (resp. $L^+(\cG_{_{\mathbb I}})$).  \end{notat}

\section{A Bruhat-Tits group scheme on the wonderful compactification} \label{btgpX}

\subsubsection{The structure of ${\bf X}$}\label{earlyloghomo}
Let ${\bf X} := \ol{G_{_{_{\text{ad}}}}}$ be the wonderful compactification of $G_{_{_{\text{ad}}}}$. 
 Let $\{ D_{\alpha} | \alpha \in S \}$ denote the irreducible smooth {\em divisors} of ${\bf X}$.  Let $D:= \cup_{{\alpha \in S}} D_\alpha$. Then $X \setminus G_{_{_{\text{ad}}}} = D$.   The pair $({\bf X},D)$ is the primary example of a  $(G_{_{_{\text{ad}}}} \times G_{_{_{\text{ad}}}})$-homogenous pair  \cite[Brion]{brion}. 
Let $P_{_{I}}$ be the standard parabolic subgroup defined by subsets $I \subset S$ , the notation being such that the Levi subgroup $L_{_{I,ad}}$ containing $T_{_{\text{ad}}}$ has root system with basis $S \setminus I$. 
Recall that the $(G_{_{_{\text{ad}}}} \times G_{_{_{\text{ad}}}})$-orbits in $\bf X$ are indexed by subsets $I \subset S$ and have the following description: $$Z_{_{I}} = (G_{_{_{\text{ad}}}} \times G_{_{_{\text{ad}}}} 
) \times _{_{P_{_{I}} \times P^-_{_{I}}}} L_{_{I,ad}}.$$ Then by \cite[Proposition A1]{brion1}, each $Z_{_{I}}$ contains a unique base point $z_{_{I}}$ such that $(B \times B^-).z_{_{I}}$ is dense in $Z_{_{I}}$ and there is a $1$-PS $\lambda$ of $T$ satisfying  $P_{_{I}} = P(\lambda)$ and $\lim_{_{t \to 0}} \lambda(t) = z_{_{I}}$. The closures of these $\lambda$ define curves $C_{_I} \subset {\bf X}$ which meet the strata $Z_{_{I}}$ transversally at $z_{_{I}}$. In particular, if $I = \{\alpha\}$ a singleton, then the divisor $D_{_\alpha}$ is the orbit closure $\bar{Z}_{_{I}}$ and the $1$-PS can be taken to be the fundamental co-weight $\omega^{\vee}_\alpha$. The closure of the $1$-PS $\omega^{\vee}_\alpha:\mathbb G_{_m} \to G_{_{\text{ad}}}$ defines the curve $C_{_\alpha} \subset \bf X$ transversal to the divisor $D_{_\alpha}$ at the point $z_{_\alpha}$.    
 
\subsubsection{The local toric structure of ${\bf X}$} Let ${\bf Y}:= \overline{T_{_{\text{ad}}}}$ be the closure of $T_{_{\text{ad}}}$ in ${\bf X}$. Recall that ${\bf Y}$ is a projective toric variety associated to the fan of Weyl chambers. In what follows, we will mostly work with the affine toric embedding ${\bf Y}_{_0} := \overline{T_{_{ad,0}}} \simeq {\mathbb A}^{^{\ell}}$ which is the toric variety  associated to the negative Weyl chamber. Let $U$ (resp $U^{^-}$) be the unipotent radical (resp. its opposite) of the Borel subgroup $B \subset G$. We also recall that one may identify $U \times U^{^-} \times {\bf Y}_{_0}$ with an open subset of $\bf X$ and moreover the $G \times G$-translates of $U \times U^{^-} \times {\bf Y}_{_0}$ covers the whole of $\bf X$.

\subsubsection{Construction of a Lie-algebra bundle $\mathcal{R}$ on ${\bf Y}_{_0}$}  \label{constructionR}
The aim of this section is to construct a Lie algebra bundle on  ${\bf Y}_{_0}$ (see \eqref{gpschLiestab}), which may be termed ``parahoric". We close by making a similar  construction  on $\bf Y$ and $\bf X$. The  construction is motivated by the structure of  the kernel of the Tits fibration in \cite{brion}.

We begin with a couple of elementary  lemmas which should be well known.
\begin{lem}\label{langton} Let $E$ be a locally free sheaf on an irreducible smooth scheme $X$. Let $\xi \in X$ be the generic point and let $W \subset E_{_{\xi}}$ be an $\mathcal O_{_{\xi}}$-submodule. Then there exists a unique coherent subsheaf $F \subset E$ such that $F_{_{\xi}} = W$ and $Q := Coker(F \hra E)$ is torsion-free. Moreover $F$ is a reflexive sheaf. \end{lem} \begin{proof} Define $\tilde{F}$ on the affine open $U \subset X$ by $\tilde{F}(U) := E(U) \cap W$.  Then it is easily seen that $\tilde{F}$ defines a coherent subsheaf $F \subset E$ and it is the maximal coherent subsheaf of $E$ whose fiber over $\xi$ is $W$. To check $Q$ is torsion-free, let $T$ be the torsion submodule of $Q$. Let $K := Ker(E \to Q/T)$. Then since $\xi \notin Supp(T)$, so $K_{_{\xi}} = W$ and hence, by the maximality of $F$ we have $K = F$. Since $E$ is locally free,  we have $F^{\vee \vee}/F \hra E/F$. But since $F^{\vee \vee}/F$ is only torsion, and $E/F=Q$ is torsion-free, it follows that $F$ is automatically a reflexive sheaf.\end{proof}
\begin{lem}\label{langton1} Let $X$ be as above and $i:U \hookrightarrow X$ an open subset such that $X \setminus U$ has codimension $\geq 2$ in $X$. Let $F_{_{U}}$ be a reflexive sheaf on $U$. Then $i_{_*}(F_{_U})$ is a reflexive sheaf on $X$ which extends $F_{_U}$. \end{lem}
\begin{proof}. By \cite[Corollaire 9.4.8]{ega}, there exists a coherent $\mathcal O_{_X}$-module $F_1$ such that $F_1|_{_{U}} \simeq F_{_U}$. Set $F := F_1^{^{**}}$ to be the double dual. Then  $F$ is reflexive and also since $F_{_U}$ is reflexive, $F|_{_{U}} \simeq F_{_U}$. Hence, we have  $i_{_{*}}(F_{_U}) = F$ \cite[Proposition 1.6]{hartshorne}. \end{proof}

Let $\lambda= \sum_{\alpha \in S} k_{_\alpha} \omega_{_\alpha}^{{\vee}}$ be a dominant $1$-PS of $T_{_{\text{ad}}}$. It defines the curve $C_{_{\lambda}} \subset \bf Y$.  When the $k_{_\alpha}$ are constrained to be  in $\{0,1\}$, then these curves are the {\it standard curves} $
C_{_I}, I \subset S$ considered in \S\eqref{earlyloghomo}, which meet the strata transversally at the points $z_{_I}$. We call these dominant $\lambda$'s as {\em standard}. In particular, the curve defined by $\omega_{_\alpha}^{{\vee}}$ meets the divisor $H_{_{\alpha}}$ transversally at $z_{_\alpha}$. 
The formal neighbourhood of the closed point $z_{_\lambda}$ in $C_{_\lambda}$ identifies  with the spectrum  $U_{_{\lambda}}:=\spec (A_{_{\lambda}} )$ of  $A_{_{\lambda}} = k\llbracket s \rrbracket$ with quotient field $K_{_{\lambda}}=k(\!(s)\!)$.  We set
\begin{equation} \label{thetalambdal+1}
\theta_\lambda := \sum_{\alpha \in S} k_{_\alpha} \frac{\theta_{_\alpha}}{\ell+1}. 
\end{equation}

When $k_{\alpha} \in \{0,1\}$, then  $\theta_\lambda$ does not lie on the far wall   of $\mathbf{a}_0$. So $\mathfrak P^{^{st}}_{_{\theta_\lambda}}=\mathfrak P_{_{\theta_\lambda}}$ \S \eqref{sbtgpsch}.

\begin{thm} \label{gpschLiestab} There exists a canonical Lie-algebra bundle $\mathcal{R}$ on ${\bf Y}_{_0}$ which extends the trivial bundle with fiber $ \mathfrak g$ on ${T_{_{\text{ad}}}}  \subset {\bf Y}_{_0}$ and such that for $K_{_\alpha} \in \{0,1\}$ we have the identification of functors from the category of $k$-algebras to $k$-Lie-algebras:
\begin{equation}\label{loopfunct}
L^+ (\mathfrak{P}^{^{st}}_{_{\theta_\lambda}})= L^+( \mathcal R\mid_{_{U_{_{\lambda}}}}).
\end{equation}

\end{thm}
\begin{proof} With notations as in \eqref{dalpha}, consider the inclusion of lattices:
\begin{equation}\label{latticeinclusion}
\bigoplus_{_{\alpha \in S}} {\mathbb Z}.\omega_{_\alpha}^{{\vee}}
 \hookrightarrow \bigoplus_{_{\alpha \in S}} {\mathbb Z}.\frac{{\theta_{_\alpha}}}{e_{_\alpha}.(\ell + 1)}.
\end{equation}

This induces an inclusion of $k$-algebras $B_{_0} \subset B$, where
\begin{equation} B:= k[y_{\alpha},y_{\alpha}^{-1}]_{_{\alpha \in S}},~~
\\
B_{_{0}}:= k[x_{\alpha},x_{\alpha}^{-1}]_{_{\alpha \in S}}
\end{equation}
and $B_{_0} \subset B$ is given by the equations: $\{{y_{_{\alpha}}^{^{(\ell+1) d_{_\alpha}}}=x_{_{\alpha}}}\}_{_{\alpha \in S}}$. Let $
{\bf T}:=\spec(B), \text{and} T_{_{\text{ad}}}=\spec(B_{_0})$ and let $p:{\bf T} \to T_{_{\text{ad}}}$ be the natural morphism. We define the ``roots" map:
\begin{eqnarray}
\mathfrak r:{\bf T} \to T_{_{\text{ad}}}, ~~~\text{as} \\
{\mathfrak r}^{\#}\big(\big(x_{\alpha}\big)\big):=\big(y_{\alpha}\big).
\end{eqnarray}
Note that as a map between tori, $\mathfrak r$ is an isomorphism.  
We consider the map
\begin{eqnarray}\label{thebbAd}
Ad \circ \mathfrak r: {\bf T} \times \mathfrak{g} \rightarrow {\bf T} \times \mathfrak{g} \\
({\bf t},x) \mapsto \big({{\bf t}}, Ad\big(\mathfrak r({\bf t})\big)(x)\big).
\end{eqnarray}

We define the embedding of modules
\begin{equation}
j: \mathfrak{g}(B_{_0}) \hookrightarrow \mathfrak{g}(B) \stackrel{Ad \circ \mathfrak r}{\rightarrow} \mathfrak{g}(B)
\end{equation}
where the second map is the one induced by \eqref{thebbAd} on sections.  When $\mathfrak{g}(B_0)$ is viewed as sections of the  trivial bundle on $T_{_{\text{ad}}}$ with fibers $\mathfrak{g}$, then it also has a $T_{_{\text{ad}}}$-weight space decomposition.

Let $B^+  = k[y_{_{\alpha}}]_{_{\alpha \in S}}$ and $B_{_0}^+:=k[x_{_{\alpha}}]_{_{\alpha \in S}}$.  Taking intersection as Lie submodules of $\mathfrak{g}(B)$ we define the following $B_{_0}^+$-module
\begin{equation}\label{sheafofR}
\mathcal{R}(B_{_0}^+):= j(\mathfrak{g}(B_{_0})) \cap \mathfrak{g}(B^+).
\end{equation}
Observe that $\mathcal R$ is a reflexive sheaf   \eqref{langton} on the affine embedding $T_{_{\text{ad}}} \hookrightarrow {\bf Y}_{_0} := \spec(B_{_0}^+) (\simeq {\mathbb A}^{^{\ell}})$. Further, this is a sheaf of Lie algebras  with its Lie bracket induced from $\mathfrak{g}(B)$.

We now check that $\mathcal{R}$ is locally-free on ${\bf Y}_{_0}$. Observe  that the intersection \eqref{sheafofR} also respects $T_{_{\text{ad}}}$-weight space decomposition on the sections. More precisely, for a root $r \in \Phi$, since by definition we have the identification $\mathcal{R}(B_{_0}^+)_r=j(\mathfrak{g}_r(B_{_0})) \cap \mathfrak{g}_r(B^+)$, we get the following equalities:
\begin{equation}
j(\mathfrak{g}(B_{_0})) \cap \mathfrak{g}_r(B^+) = j(\mathfrak{g}_r(B_{_0})) \cap \mathfrak{g}(B^+)= \mathcal{R}(B_{_0}^+)_r.
\end{equation} 

Since $\mathcal{R}$ is reflexive, so there is an open subset $U \subset {\bf Y}_{_0}$,  whose complement is of codimension at least two such that the restriction $\mathcal{R}':= \mathcal{R}|_{_{U}}$ is locally free. 
Clearly $U$ contains $T_{_{\text{ad}}}$ and the generic points $\zeta_{_{\alpha}}$ of the divisors $H_{_{\alpha}}$. This gives a  decomposition on $\mathcal{R}'$ obtained by restriction from $\mathcal{R}$.   The locally free sheaf $\mathcal{R}'$ is a direct sum of the trivial bundle $\text{Lie}(T_{_{\text{ad}}}) \times U$ (coming from the $0$-weight space) and the invertible sheaves coming from the root decomposition. Now since invertible sheaves extend across codimension $\geq 2$, the reflexivity of $\mathcal{R}$ implies that this direct sum decomposition of $\mathcal{R}'$ extends to  ${\bf Y}_{_0}$ (see \cite[Proposition 1.6, page 126]{hartshorne}). Whence  $\mathcal{R}$ is locally free.

Let $\lambda:\mathbb G_{_m} \to T_{_{\text{ad}}}$ be a $1$-PS. This defines a rational $1$-PS $\theta_{_\lambda}:\mathbb G_{_m} \to T_{_{\text{ad}}}$. The map $\mathfrak r: {\bf T} \to T_{_{\text{ad}}}$ is abstractly an isomorphism of tori and we can therefore consider the rational $1$-PS $\boldsymbol{\theta_{_\lambda}}:\mathbb G_{_m} \to {\bf T}$ defined by $\boldsymbol{\theta_{_\lambda}} := \mathfrak r^{^{-1}} \circ \theta_{_\lambda}$. Let $p:{\bf T} \to T_{_{\text{ad}}}$ be the canonical map induced by $B_{_0} \subset B$. We observe the following:
\begin{enumerate}
\item $p \circ \boldsymbol{\theta_{_\lambda}} = \lambda$,
\item $\mathfrak r \circ \boldsymbol{\theta_{_\lambda}} = \theta_{_\lambda}.$
\end{enumerate}

Let $\lambda = \sum k_{_\alpha}.\omega_{_\alpha}^{{\vee}}$ be a  standard dominant $1$-PS of $T_{_{\text{ad}}}$.  
For example, the case $\lambda =\omega_{_\alpha}^{{\vee}}$, viewed as a $1$-PS, may be expressed in $\ell$-many coordinates as follows:
\begin{equation}
\omega_{_\alpha}^{{\vee}}(s) = (1,\ldots, s, \ldots 1)
\end{equation}
with $s$ at the coordinate corresponding to $\alpha \in S$. Thus, we may express $\lambda$ as:
\begin{equation}
\lambda(s) = \prod_{\alpha \in S} \omega_{_\alpha}^{{\vee}}(s^{^{k_{_{\alpha}}}}). 
\end{equation}
Set $f_{_\alpha} := \frac{k_{_\alpha}} {d_{_\alpha}.(\ell + 1)}$. In this case we have the expression:
\begin{equation}\label{theelam}
\boldsymbol{\theta_{_\lambda}}(s) = \prod_{\alpha \in S} \omega_{_\alpha}^{{\vee}}(s^{^{f_{_{\alpha}}}}).
\end{equation}
We now check the isomorphism $L^+({\mathcal{R}}|_{_{U_{_\lambda}}}) \simeq L^+(\mathfrak P_{_{\theta_\lambda}}^{^{st}})$ by first evaluating at $k$-valued points. By \eqref{sheafofR}, a section $\text{\cursive s}$ of $\mathcal R$ over $U_{_\lambda} = \spec A_{_\lambda}$,  is firstly given by a local section $\text{\cursive s}_{_K}$ over the generic point of $A_{_\lambda}=k\llbracket s \rrbracket$. This is firstly an element in  $j(\gfr(K))$. Observe that  this element may be written as
$$\text{\cursive s}_{_K} = Ad\big(\mathfrak r({\boldsymbol{\theta_{_\lambda}}(s)}\big)(x_{_K}))= Ad\big(\theta_{_{\lambda}}(s)\big)(x_{_{K}})).$$

%Here, $d = d_{_\lambda}.(\ell + 1)$, where $d_{_\lambda} = \text{lcm}\{k_{_\alpha}.d_{_\alpha}\}$, with $d_{_\alpha}$ as in \eqref{dalpha}.
Let $d$ be any positive integer such that $d. \theta_\lambda$ becomes a $1$-PS of $T_{_{\text{ad}}}$. 
Let $(\tilde{B}, w,L)$ be defined by taking a $d$-th root of the uniformizer $s$ of $A_{_\lambda}$ and let $\tilde{U}_{_\lambda}=\spec (\tilde{B})$.    By \eqref{theelam}, we may express $\boldsymbol{\theta_{_\lambda}}$ in terms of $w$ as  $$\boldsymbol{\theta_{_\lambda}}(s) = \prod_{\alpha \in S} \omega_{_\alpha}^{{\vee}}(w^{^{d.f_{_{\alpha}}}}).$$ In other words, $\theta_\lambda(s)=\theta_\lambda(w^d)=(d \theta_{_\lambda})(w)$ has become integral in $w$. Therefore, the membership of $\text{\cursive s}$ in $\mathfrak{g}(B^+)$ \eqref{sheafofR} gets interpreted as follows:
\begin{equation}\label{conditions1}     Ad(\theta_\lambda(s))(x_{K}) \in L \mathfrak{g} (\tilde{B}).
\end{equation} 

But this is exactly the Lie-algebra version  of the condition "$\lim_{s \rightarrow 0}$" condition for $\eta=(1,\theta)$ in the observation \eqref{observation}.
In our situation we have assumed $R=k$ and so \eqref{conditions1} is equivalent to
%(\ref{paraLiealg}) and viewing ${\mathfrak{P}}_{_{\theta_{_\alpha}}}(k) \subset \mathfrak{g}\big(k(\!(s)\!)\big)$ as in the reformulation \eqref{meaning1} the above conditions imply that  
\begin{equation}  \text{\cursive s} \in   \mathfrak{P}_{_{\theta_{_{\lambda}}}}(k).
\end{equation}
But by \eqref{relnstdnonstd}, $\mathfrak{P}_{_{\theta_{_{\lambda}}}}(k)= \mathfrak{P}_{_{\theta_{_{\lambda}}}}^{^{st}}(k)$  because $\theta_{_{\lambda}}$ lies in the alcove $\bf a_{_0}$ but not on the far wall.
Thus, we get the equality
\begin{equation}
\mathcal{R}|_{_{U_{_\lambda}}}(A_{_\lambda})= \mathfrak{P}_{_{\theta_{_{\lambda}}}}^{^{st}}(k).
\end{equation}
This proves the assertion for $k$-valued points.
The above proof goes through for all $k$-algebras because the Lie-bracket is already defined on $k$, and so is the underlying module structure. \end{proof}

\begin{rem} \label{gendomlambda} Let $\lambda$ be an arbitrary dominant $1$-PS of $T_{_{\text{ad}}}$. In general,  the coefficients $k_{_\alpha}$ in its expression in terms of the $\omega_{_\alpha}^{\vee}$ need not be such that $\sum k_{_\alpha}/(\ell +1) < 1$. Then we get the identification $L^+ (\mathfrak{P}_{_{\theta_\lambda}})= L^+( \mathcal R\mid_{_{U_{_{\lambda}}}})$ without the {\em standardness} of the parahoric.
\end{rem}

\begin{rem} \label{torickawamata} In the setting of Theorem \ref{gpschLiestab}, we may view $p:\spec(B^{^+}) \to \spec(B_{_0}^{^+})$ as a ramified covering space of affine toric varieties induced by the inclusion \eqref{latticeinclusion} of lattices. The Galois group for this covering is the dual of the quotient of lattices.  The computation in \ref{gpschLiestab} can be seen in the light of \cite{base}, in the sense that via an "invariant direct image" process, one is able to recover the complete data of all {\em standard} parahoric Lie algebras from this explicit Kawamata cover.

\end{rem}

\subsubsection{Construction of a Lie-algebra bundle $\mathcal{R}$ on ${\bf Y}$} 
In this subsection, we deduce the existence of Lie algebra bundles on the projective non-singular toric variety $\bf Y$  with the classifying properties. The variety $\bf Y$ is the toric variety for ${T_{_{\text{ad}}}}$ with fan consisting of the Weyl chambers. The complement ${\bf Y} \setminus {T_{_{\text{ad}}}}$ is a union  of translates of the hyperplanes $H_{_\alpha} \subset {\bf Y}_{_0}$ by the Weyl group  $W$. Thus, for each $w \in W$, we can take the locally free sheaf $\mathcal R$ of Lie-algebras on ${\bf Y}_{_0}$ and its $w$-translate $w.\mathcal R$ on ${\bf Y}_{_w} := w{\bf Y}_{_0}$ in $\bf Y$. 

Let $Y'' \subset {\bf Y}_{_0}$ be the open subset consisting of
the open orbit $T_{_\text{ad}}$ and of the orbits of codimension 1.
In other words, $Y''$ is obtained from ${\bf Y}_{_0}$ by removing
all the $T_{_\text{ad}}$-orbits of codimension at least 2.  The data above defines a locally free sheaf of Lie algebras on $\cup_{_{w \in W}} w. Y''$ because $w.Y'' \cap Y''$ is either $Y''$ or $T_{_\text{ad}}$. Using \eqref{langton1}, we get a reflexive sheaf which by abuse of notation can be called $\mathcal R$ on the whole of ${\bf Y}$. That $\mathcal R$ on $\bf Y$ is locally free is immediate since its restrictions to the open cover given by ${\bf Y}_{_0}$ and its $W$-translates  is locally free. The Lie algebra structure also extends since it is already there on an open subset with complement of codimension $\geq 2$. Thus, we conclude:
\begin{Cor} \label{gpschLiestab1} There exists a canonical Lie-algebra bundle $\mathcal{R}$ on ${\bf Y}$ which extends the trivial bundle with fiber $ \mathfrak g$ on ${T_{_{\text{ad}}}}  \subset {\bf Y}$ with properties as in \eqref{gpschLiestab} at the translates of the hyperplanes $H_{_\alpha}$. 
\end{Cor}

\subsubsection{Weil restrictions  and \text{Lie} algebras}
Let $X$ be an arbitrary $k$-scheme.
For an affine (or possibly ind-affine) group scheme $\cH \rightarrow X$, we denote $\text{Lie}(\cH)$ the sheaf of Lie-algebras on $X$ whose sections on $U \rightarrow X$ are given by
\begin{equation} \label{Lieaffinegpsch} \text{Lie} (\cH)(U)= \text{ker}(\cH(U \times k[\epsilon]) \rightarrow \cH(U)).
\end{equation}

\begin{lem} \label{Weilrestriction} Let $p: \tilde{X} \ra X$ be a finite flat map of noetherian schemes. Let $Res_{\tilde{X}/X}$ denote the ``Weil restriction of scalars" functor. Let $\cH  \ra \tilde{X}$ be an affine group scheme. Then we have a natural isomorphism
\begin{equation}
{\text{Lie}}(Res_{\tilde{X}/X} \cH) \simeq Res_{\tilde{X}/X} {\text{Lie}}(\cH).
\end{equation}  
When $p$ is also Galois with Galois group $\Gamma$, then in characteristic $0$, we have 
\begin{equation}
{\text{Lie}}((Res_{\tilde{X}/X} \cH)^{\Gamma}) \simeq (Res_{\tilde{X}/X} {\text{Lie}}(\cH))^{\Gamma}.
\end{equation}
\end{lem}
\begin{proof} See \cite[Page 533]{cgp} and \cite[3.1, page 293]{edix} respectively.
\end{proof}

\subsubsection{Restriction of $\mathcal{R}$ to infinitesimal standard curves $U_{I}$} Let $U_{_I}$ be the formal neighbourhood of the standard curve $C_{_I}$ \S \eqref{constructionR} at its closed point. We call the corresponding dominant $1$-PS $\lambda$ of $T_{_{\text{ad}}}$ to be {\it standard}. Let $\theta_{_{\lambda}}$ be the point in $\mathbf{a}_0$ \eqref{thetalambdal+1}.

Recall that by \cite[Proposition 5.1.2]{base} for each $\theta_{_\lambda}$ there exists a ramified cover $q_{_\lambda}: U'_{_\lambda} \to  U_{_\lambda}$ (of ramification index $d$ as in \eqref{dtheta}) together with a $\Gamma_{_\lambda}$-equivariant $G$-torsor $E_{_\lambda}$ such that the adjoint group  scheme $\cH_{_\lambda}=E_{_\lambda}(G)$ has simply-connected fibers isomorphic to $G$ and we have the identification of $U_{_\lambda}$-group schemes:
\beqa
\big(\text{Res}_{_{U'/U}}(\cH_{_\lambda})\big)^{^{\Gamma_{_\lambda}}} \simeq  \cG^{^{st}}_{_{\theta_\lambda}}.
\eeqa

We therefore get the following useful corollary to \eqref{gpschLiestab}.
\begin{Cor}\label{gpschLiestab3} For any standard dominant $1$-PS $\lambda$ of $T_{_{ad}}$, with notations as above we have an isomorphism as sheaves of Lie algebras:

\begin{equation} \label{Liestrendow} \mathcal{R}|_{_{U_{_\lambda}}} \simeq q^{^{\Gamma}}_{_{\lambda,*}}(E_{_\lambda}(\mathfrak g)).
\end{equation}
 
\end{Cor}
\begin{proof} This is an immediate consequence of  \eqref{gpschLiestab}  and \cite[Proposition 5.1.2]{base}. \end{proof}

All results in this subsection generalize suitably to non-standard curves corresponding to dominant $1$-PS in $T_{_{\text{ad}}}$ by \eqref{gendomlambda}.

\subsubsection{The Lie algebra bundle $\mathcal R$ on $\bf X$}\label{atthehyper}
Let $\{ D_{\alpha} | \alpha \in S \}$ denote the irreducible smooth {\em boundary divisors} of ${\bf X}$. Set $H_{_\alpha} : = D_{_\alpha} \cap {\bf Y}_{_0}$ for each $\alpha$. Recall that $Z:= {\bf Y}_{_0} \setminus T_{_{\text{ad}}}$ is a union $\cup_{_{\alpha \in S}} H_{_\alpha}$  which are $\ell $ smooth {\em hyperplanes} meeting at simple normal crossings.     For $\alpha \in S$, let $\zeta_{_\alpha}$'s denote the generic points of the divisors $H_{_\alpha}$'s. Let  
\begin{equation} \label{Aalpha0}
A_{_\alpha}= \mathcal O_{_{{{\bf Y}_{_0}}, \zeta_{_{\alpha}}}}
\end{equation} 
be the dvr's obtained by localizing at the height $1$-primes given by the $\zeta_{_\alpha}$'s and let $Y_{_\alpha} := \spec(A_{_\alpha})$. Base changing by the local morphism $Y_{_\alpha} \to \bf Y_{_0}$, we have a Lie algebra bundle $\mathcal R|_{_{Y_{_\alpha}}}$ for each $\alpha$. Moreover,  the Lie algebra bundle $\mathcal R$ on ${\bf Y}_{_0}$  gives  canonical gluing data to glue $\mathcal R|_{_{Y_{_\alpha}}}$ with the trivial bundle $\mathfrak g \times T_{_{\text{ad}}}$.

Let  $\xi_{_\alpha} \in D_{_\alpha}, \alpha \in S$ denote the generic points of the divisors $D_{_\alpha}$'s and let  
\begin{equation}
B_{_\alpha}:= \mathcal O_{_{{\bf X}, \xi_{_{\alpha}}}}
\end{equation} 
be the dvr's obtained by localizing at the height $1$-primes given by the $\xi_{_\alpha}$'s.
By a transport of structures, for each $\alpha$ the gluing datum of $\mathcal R|_{_{Y_{_\alpha} \cap T_{_{\text{ad}}}}}$, gives a Lie-algebra bundle gluing datum on $\spec(B_{_\alpha}) \cap G_{_{\text{ad}}}$.  Thus, the gluing data of the bundle $\mathcal R$ on ${\bf Y}_{_0}$ at the  the $\zeta_{_\alpha}$'s can be now used to extend the trivial bundle $\mathfrak g \times G_{_{\text{ad}}}$ to the $\xi_{_\alpha}$'s as a locally free sheaf of Lie algebras. The rest of the proof is as for $\bf Y$, together with the observation that the $G \times G$-translates of the open subset $U \times U^{^-} \times {\bf Y}_{_0}$ cover $\bf X$, where the bundle is simply the pull-back from ${\bf Y}_{_0}$. Thus we have: 
\begin{Cor} \label{gpschLiestab2} There exists a canonical Lie-algebra bundle $\mathcal{R}_{_{\bf X}}$ on ${\bf X}$ which extends the trivial bundle with fiber $ \mathfrak g$ on ${G_{_{\text{ad}}}}  \subset {\bf X}$ with properties as in \eqref{gpschLiestab}. 
\end{Cor}

\subsubsection{Bruhat-Tits group scheme ${\mathcal G}_{_{\bf X}}^{^{\varpi}}$ on $\bf X$} 
  Let $L_{_\alpha}$ be the quotient field of $B_{_\alpha}$. Let $X_{_\alpha}:= \spec(B_{_\alpha}) $. By \eqref{gpschLiestab3}, for each $\alpha$ there exist  a $\Gamma_{_\alpha}$-equivariant $G$-torsor $E_{\alpha}$ on a ramified cover $q_{_\alpha}: X'_{_\alpha} \to  X_{_\alpha}$ such that the adjoint group  scheme $\cH_{_\alpha}=E_{\alpha}(G)$ has simply-connected fibers isomorphic to $G$ and we have the identification of $B_{_\alpha}$-group schemes:
\beqa\label{invdi}
\big(\text{Res}_{_{X'_{_\alpha}/X_{_\alpha}}}(\cH_{_\alpha})\big)^{^{\Gamma_{_\alpha}}} \simeq  \cG^{^{st}}_{_{\theta{_{_\alpha}}}}.
\eeqa
Before stating  the main result of this section,  we make a few remarks which might help the reader. The basic underlying principle in these constructions
is that the combinatorial data encoded in the triple consisting of the Weyl chamber, the fan of Weyl chambers, and the Tits building,   is geometrically replicated by the inclusion ${\bf Y_{_0}} \subset {\bf Y} \subset \bf X$. The "wonderful" Bruhat-Tits group scheme which arises on $\bf X$ has its local Weyl-chamber model on the affine toric variety $\bf Y_{_0}$. Indeed, in this case, the Kawamata cover is even explicit \eqref{torickawamata}. In particular, the theorem below, can be executed for ${\bf Y}_{_0}$ but this will give the group scheme associated to the data coming from the Weyl chamber alone.  
\begin{thm}\label{btoverx}  There exists an affine ``wonderful" Bruhat-Tits group scheme ${\mathcal G}_{_{\bf X}}^{^{\varpi}}$ on $\bf X$ satisfying the following classifying properties. 
\begin{enumerate} 
\item There is an identification of the Lie-algebra bundles $\text{Lie}({\mathcal G}_{_{\bf X}}^{^{\varpi}}) \simeq \mathcal{R}_{_{\bf X}}$.

\item For $\emptyset \neq I \subset S$ the restriction of ${\mathcal G}_{_{{\bf X}}}^{^{\varpi}}$ to  the formal neighbourhood $U_{_{z_{_I}}}$ of $z_{_I}$ in $C_{_I}$ \S \eqref{constructionR} is isomorphic to the standard Bruhat-Tits group scheme $\cG^{^{st}}_{_{I}}$ \S \eqref{sbtgpsch}.
\end{enumerate}  
\end{thm}

\begin{proof}
Let $X_{_\alpha}$ be as above. Note that we can identify the open subset $\spec (L_{_\alpha})$ with  ${G_{_{\text{ad}}}} \cap X_{_\alpha}$. Set 
\begin{equation} X':=G_{_{\text{ad}}} \cup_{\alpha \in S} X_{_\alpha}.
\end{equation} Hence ${\bf X} \setminus X'$ is a colimit of closed subschemes of $\bf X$ of codimension at least $2$. Consider the {\em fpqc} morphism:
\begin{equation}
 G_{_{_{\text{ad}}}} \bigsqcup_{\alpha \in S } X_{_\alpha} \to {X'}.
\end{equation}
We restrict $\mathcal R = \mathcal{R}_{_{\bf X}}$ further to $X'$. By \eqref{gpschLiestab2} ,  over the open subset $G_{_{\text{ad}}} \subset X'$, we have $\mathcal R \simeq G_{ad} \times \mathfrak{g}$. Further, the transition functions of $\mathcal R$ on the  intersections $\spec(L_{_\alpha}):=G_{_{_{\text{ad}}}} \cap \spec(B_{_\alpha})$ take values in $\text{Aut}(\gfr)$. Since $\text{Aut}(\gfr)=\text{Aut}(G)$, it follows that the effective descent datum provided by pulling back $\mathcal R$ to $G_{_{_{\text{ad}}}} \bigsqcup_{\alpha \in S } X_{_\alpha}$ gives a descent datum to glue the trivial group  scheme $G \times G_{_{_{\text{ad}}}}$ on $G_{_{_{\text{ad}}}}$ with $\cG^{^{st}}_{_{\alpha}}$ on $ \spec(B_{_\alpha})$ along $\spec(L_{_\alpha})$. Since the group schemes are affine, this descent datum is also effective. In other words, 
we get the group scheme
\begin{equation} 
\cG^{\circ} \ra X'.
\end{equation}
This can also be seen without ``descent" theory as follows. Since the group schemes $\cG^{^{st}}_{_{\alpha}}$ on $ \spec(B_{_\alpha})$ are of finite type, they can be extended to an affine subscheme $\bf X_{_f}$. By a further shrinking of this neighbourhood of $\xi_{_\alpha}$ one can glue it to $G \times G_{_{_{\text{ad}}}}$ along the intersection. So we can think of $X'$ as an honest open subset of $X$ with complement of codimension at least $2$.

By \eqref{gpschLiestab3}, the Lie algebra bundle $\mathcal R$ gets canonical parabolic structures at the generic points $\xi_{_\alpha}$ of the divisors $\{D_{_\alpha}\}_{_{\alpha \in S}}$.
For a parabolic vector bundle  with prescribed rational weights such as $\mathcal R^{^{}}$, by \cite{biswas} (see \S\ref{parabstuff}) we get the following data:
\begin{itemize} 
\item a global Kawamata cover (\ref{kawa}) $p: Z \ra {\bf X}$ ramified over $D$ with ramification prescribed by the weights $\{d_{_\alpha}\}$ \eqref{dalpha}, with Galois group $\Gamma$ which ``realizes the local ramified covers $q_{_\alpha}$ at the points $\xi_{_\alpha}$" i.e. the isotropy subgroup of $\Gamma$ at $\xi_{_\alpha}$ is $\Gamma_{_{\alpha}}$ and 
\item an equivariant vector bundle $V$ on $Z$ such that 
\begin{equation}
p_{_*}^{\Gamma} (V) \simeq \mathcal R^{^{}}.
\end{equation}
\end{itemize}

Let $Z':=  p^{-1}(X')$ and $V' := V|_{_{Z'}}$. Gluing the trivial $G$-torsor  with the $\{E_{\alpha} \}$ by the transition functions of $V'$, we make a $\Gamma$-equivariant principal $\text{Aut}(G)$-torsor $\cE^{\circ}$. Its associated group scheme $\cE^{\circ} \times ^{^{\text{Aut}(G)}} G = \cH^\circ$ is a now group scheme on $Z'$ with simply-connected fibers  $G$, albeit with transition functions in $\text{Aut}(G)$. Moreover, we have 
\begin{equation}
\big(\text{Res}_{_{Z'/X'}}(\cH^{\circ})\big)^{^{\Gamma}}  \simeq \cG^{\circ}.
\end{equation}
Further, as locally-free sheaves, we also have
\beqa
\text{Lie}(\cH^{\circ})=V'.
\eeqa
We transport the structure of Lie-bracket from $\text{Lie}(\cH^\circ)$ to $V'$ on $Z'$. This is non-degenerate everywhere on $Z'$ since $\cH^\circ$ is the adjoint group scheme of $\cE^{\circ}$. Hence $V'$ has fiber type $\mathfrak g$. Since $\text{codim}(Z \setminus Z') \geq 2$, the Lie bracket $[.,.]$ on $V'$ extends to a Lie bracket on the locally free sheaf $V$ with the Killing form being non-degenerate on the whole of $Z$. {\sl In other words $V$ is now a locally free sheaf of Lie algebras on the whole of $Z$ with semisimple fibres; these fibres are isomorphic to the Lie algebra $\mathfrak g$ by the rigidity of semisimple Lie algebras}.

We now wish to extend $\cE^{\circ}$ as an equivariant $\text{Aut}(G)$-torsor $\cE$ on the whole of $Z$ such that its associated Lie-algebra bundle $\cE(\mathfrak g) := \cE \times ^{^{\text{Aut}(G)}} \mathfrak g$ becomes isomorphic $V$. Since $G$ is simply-connected, making the  identification $\text{Aut}(\gfr)=\text{Aut}(G)$, we see that the transition functions of $V$ give the gluing for defining $\cE$ on $Z$ satisfying our desired requirements.

Let $\cH := \cE \times ^{^{\text{Aut}(G)}} G$ denote the group scheme associated to the adjoint group scheme of $\cE$. This is an equivariant group scheme on $Z$ and we define
\beqa\label{justnow}
{\mathcal G}_{_{\bf X}}^{^{\varpi}}:= \big(\text{Res}_{_{Z/{\bf X}}}(\cH)\big)^{^{\Gamma}}.
\eeqa

\noindent
Then ${\mathcal G}_{_{\bf X}}^{^{\varpi}}|_{{X'}}=\cG^\circ$. Since by Lemma \eqref{Weilrestriction}, the functor ``invariant direct image" commutes with taking Lie algebras, we moreover get isomorphisms of locally-free sheaves of Lie algebras
\begin{equation}\label{liestr's}
\text{Lie}({\mathcal G}_{_{\bf X}}^{^{\varpi}}) \simeq p_{_*}^{\Gamma} (\cE(\mathfrak g)) \simeq  \mathcal R^{^{}}.
\end{equation}
This proves the first claim in the theorem.

Finally, let us verify the classifying property of the group scheme $\cH$. For $\alpha \in S$, at the closed points $z_{_\alpha}$ of the curves $\spec(A_{_\alpha})$ we have $\cH|{_{\spec(A_{_\alpha})}} \simeq \cH_{_{\alpha}}$. The classifying property is tautologically valid here because this was designed expressly  in \eqref{invdi}.

For the closed points $z_{_\lambda}$ of $C_{_\lambda}$ corresponding to strata of lower dimension, we proceed as follows. Consider the base change of the Kawamata cover $p: Z \ra {\bf X}$ to the curve $C_{_\lambda} \subset {\bf X}$ and further to the formal neighbourhood $U_{_{z_{_\lambda}}} \subset C_{_\lambda}$ of $z_{_\lambda}$.

Let $p_{_z}:W_{_{z}} \to U_{_z}$ be the restriction of $p$ to a connected component of $Z \times_{_{U_{_z}}} {\bf X}$. Then, $p_{_z}$ gives a Galois cover with Galois group some cyclic group  $\mu_{_d} \subset \Gamma$ of order $d$.

By \eqref{justnow}, the restriction ${\mathcal G}_{_{{\bf Y}}}^{^{\varpi}}|_{_{U_{_{z}}}}$ is the ``invariant direct image" of  $\cH|_{_{W_{_{z}}}}$. 
Also, we have the isomorphism $\mathcal R|_{_{U_{_{z}}}} \simeq p_{_*}^{\mu_{_d}}(\cE(\mathfrak g)_{_{W_{_{z}}}})$. Further, by \eqref{gpschLiestab3},  $L^+(\mathcal R|_{_{U_{_{z_{_\lambda}}}}}) \simeq L^+(\mathfrak P^{^{std}}_{_{\theta_{_\lambda}}})$.

By \eqref{gpschLiestab3}, we have another ramified cover $q_{_\lambda}:U'_{_z} \to U_{_z}$, but now with Galois group $\Gamma_{_\lambda}$ and {\sl an equivariant $(\Gamma_{_\lambda},G)$-torsor $E_{_\lambda}$ on $U'_{_z}$}, such that $q^{^{\Gamma_{_\lambda}}}_{_{*}}(E_{_\lambda}(G)) \simeq \cG^{^{st}}_{_{\theta_{_\lambda}}}$ and 
$q^{^{\Gamma_{_\lambda}}}_{_{*}}(E_{_\lambda}(\mathfrak g)) \simeq \mathcal R|_{_{U_{_{z_{_\lambda}}}}}$.

To finish the proof, we need to show the following isomorphism of group schemes:
\begin{equation}\label{appadi}
{\mathcal G}_{_{{\bf X}}}^{^{\varpi}}|_{_{U_{_{z}}}} \simeq  \cG^{^{st}}_{_{\theta_{_\lambda}}}.
\end{equation}

 Over a common cover of $U'_{_z}$ and $W_{_z}$ (which will continue to be a Kawamata cover of $U_{_z}$), we can identify the pull-back of Lie sheaves $E_{_\lambda}(\mathfrak g)$ and $\cE(\mathfrak g) |_{_{W_{_{z}}}}$ as equivariant Lie sheaves since both give invariant direct images isomorphic to $\mathcal R|_{_{U_{_{z_{_\lambda}}}}}$. Therefore, on the same common cover we have an identification of pull-backs of $E_{_{\lambda}} \times^{G} \mathfrak{g}$ with $\cE(\mathfrak g) |_{_{W_{_{z}}}}=\mathcal{E} \times^{Aut(G)} \mathfrak{g} |_{_{W_{_{z}}}}$. Hence we have an identification of pull-backs of $E_{_\lambda} \times^{G} G$ with $\mathcal{E} \times^{Aut(G)} G |_{_{W_{_{z}}}}$ as equivariant group schemes with simply-connected fibers. Now the invariant direct image of the first group scheme is a standard parahoric group scheme because the curve $C_{_\lambda}$ is standard. On the other hand, the second group scheme is $\cH|_{_{W_{_z}}}$ whose invariant direct image by construction is  ${\mathcal G}_{_{{\bf X}}}^{^{\varpi}}|_{_{U_{_{z}}}}$. Thus, we have proven   \eqref{appadi}.

\end{proof}

\begin{rem}\label{pabsolute} The above construction goes through without change over an algebraically closed field $k$ of characteristic $p$ coprime to the $d_{_\alpha}$'s \eqref{dalpha}. The existence of $\bf X$ is known from the works of Strickland \cite{El} and De Concini-Springer \cite{DS} and Kawamata covering works under the above conditions on characteristics.
\end{rem}
\section{The Weyl alcove and apartment case}\label{loopcase}

We continue to use the notations as in previous sections.  Let $G^{^{\text{aff}}}$ denote the Kac-Moody group associated to the affine Dynkin diagram of $G$. Recall that $G^{^{\text{aff}}}$ is given by a  central extension of $L^{\ltimes}G$ by $\GG_m$. Analogous to  the wonderful compactification of $G_{_{_{\text{ad}}}}$, P. Solis in \cite{solis} has constructed a wonderful embedding ${\bf X}^{^{\text{aff}}}$ for $G^{^{\text{aff}}}_{_{\text{ad}}}:=G^{^{\text{aff}}}/Z(G^{^{\text{aff}}})=\GG_m \ltimes LG/Z(G)$.  It is an ind-scheme containing $G^{^{\text{aff}}}_{_{\text{ad}}}$ as a dense open ind-scheme and carrying an equivariant action of $L^{\ltimes} G \times L^{\ltimes} G$. 

Let $T_{_{\text{ad}}} := T/Z(G)$, $T^{\ltimes}_{_{\text{ad}}}:= \GG_m \times T_{_{\text{ad}}} \subset G^{^{\text{aff}}}_{_{\text{ad}}}$, where $\GG_m$ is the rotational torus. In ${\bf X}^{^{\text{aff}}}$, the closure ${\bf Y}^{^{\text{aff}}}:=\ol{T^{\ltimes}_{_{\text{ad}}}}$ gives  a  torus-embedding. It is  covered by the  affine Weyl group $W^{^{\text{aff}}}$-translates  of the affine torus embedding ${\bf Y}_{_{0}}^{^{\text{aff}}} := \ol{T^{\ltimes}_{_{\text{ad},0}}} \simeq  \mathbb{A}^{\ell+1}$ given by the negative Weyl alcove.

\subsubsection{On the torus-embedding ${\bf Y}_{_{0}}^{^{\text{aff}}}$} Recall that $Z:= {\bf Y}_{_{0}}^{^{\text{aff}}} \setminus T^{\ltimes}_{_{\text{ad}}}$ is a union $\cup_{_{\alpha \in \mathbb S}} H_{_\alpha}$ of $\ell+1$ many standard coordinate hyperplanes meeting at normal crossings.   For $\alpha \in \mathbb S$, let the $\zeta_{_\alpha}$'s denote the generic points of the divisors $H_{_\alpha}$'s. Let  
\begin{equation} \label{Aalpha}
A_{_\alpha}= \mathcal O_{_{{\bf Y}_{_{0}}^{^{\text{aff}}}, \zeta_{_{\alpha}}}}
\end{equation} 
be the dvr's obtained by localizing at the height $1$-primes given by the $\zeta_{_\alpha}$'s. Let $K_{_\alpha}$ be the quotient field of $A_{_\alpha}$. Let $Y_{_\alpha}:= \spec(A_{_\alpha}) $. Note that we can identify the open subset $\spec (K_{_\alpha})$ with  ${T^{\ltimes}_{_{\text{ad}}}} \cap Y_{_\alpha}$. Let 
\begin{equation}\label{yprime} Y':= {T^{\ltimes}_{_{\text{ad}}}} \cup_{\alpha } Y_{_\alpha}.
\end{equation}
The complement ${\bf Y}_{_{0}}^{^{\text{aff}}} \setminus Y'$ can again be realized as a colimit of open subsets of ${\bf Y}_{_{0}}^{^{\text{aff}}}$ whose   codimension is at least $2$ in ${\bf Y}_{_{0}}^{^{\text{aff}}}$.

\subsubsection{Construction of a finite-dimensional Lie algebra bundle $J$ on ${\bf Y}_{_0}^{^{\text{aff}}}$ together with parabolic structures} This construction is exactly analogous to the construction of $\mathcal{R}$ on $\bf Y$ in \S \ref{constructionR}. We let $T^{\ltimes}_{_{\text{ad}}}$ and $\mathbb{S}$ play the role of $T_{_{\text{ad}}}$ and $S$.

More precisely, let $\lambda= \sum_{\alpha \in \mathbb{S}} k_{_\alpha} \omega_{_\alpha}^{{\vee}}$ be a non-zero dominant $1$-PS of $T_{_{\text{ad}}}^{\ltimes}$. We set
\begin{equation}
\eta_\lambda := \sum_{\alpha \in \mathbb S} k_{_\alpha} \frac{(1,\theta_{_\alpha})}{\ell+1} \quad \text{and} \quad \theta_\lambda :=  \sum_{\alpha \in \mathbb S} \frac{k_{_\alpha}}{\sum_{_{\alpha \in \mathbb{S}} k_{_{\alpha}}}} \theta_{_\alpha} 
\end{equation}
and $C_{_\lambda}$ the curve in ${\bf Y}^{^{\text{aff}}}$ defined by $\eta_\lambda$ and $U_{_\lambda}$ the formal neighbourhood of the closed point of $z_{_\lambda}$ of $C_{_\lambda}$. So if we substitute the rational number $a$  of \S \ref{loopgpsec} by the number $\frac{\sum_{\alpha \in \mathbb{S}} k_{_{\alpha}}}{\ell+1}$, we have the parahoric group identification
\begin{equation} \mathfrak P_{_{\eta_\lambda}}=\mathfrak P_{_{\theta_\lambda}}.
\end{equation}

We may prove the following theorem exactly like \eqref{gpschLiestab}. {\sl There we viewed the standard alcove as a cone over the far wall}. Note here that unlike the situation in \eqref{gpschLiestab}, we do not expect to get {\em standard} parahoric structures down all the strata. However, for strata contained in the divisor associated to $\alpha_{_0} \in \mathbb S$, the parahorics which occur will be standard as before. 
\begin{thm} \label{liealgbunonyaff}  There exists a canonical Lie-algebra bundle $J$ on ${\bf Y}_{_0}^{^{\text{aff}}}$ which extends the trivial bundle with fiber $ \mathfrak g$ on ${T^{\ltimes}_{_{\text{ad}}}}  \subset {\bf Y}^{^{\text{aff}}}$ and such that for $k_{_\alpha} \in \{0,1\}$ we have the identification of functors from the category of $k$-algebras to $k$-Lie-algebras:
\begin{equation}
L^+ (\mathfrak{P}_{_{\eta_{_\lambda}}})= L^+( \mathcal R\mid_{_{U_{_{\lambda}}}}).
\end{equation}

\end{thm}

\begin{Cor}\label{canparstr} The Lie algebra bundle $J_{_{{\bf Y}^{^{\text{aff}}}}}$ gets canonical parabolic structures (\S\ref{parabstuff}) at the generic points $\xi_{_\alpha}$ of the $W^{^{\text{aff}}}$-translates of the divisors $H_{_\alpha} \subset {\bf Y}_{_0}^{^{\text{aff}}} $, $\alpha \in \mathbb S$. 
\end{Cor} 
\begin{proof} We prescribe ramification indices $d_\alpha$ \eqref{dalpha} on the divisor $H_{_\alpha}$. Then as in the proof of \eqref{gpschLiestab1}, the identification \eqref{Liestrendow} of the Lie algebra structures of $J_{_{{\bf Y}^{^{\text{aff}}}}}$ and the parahoric Lie algebra structures on the localizations of the generic points of  $H_{_\alpha}$ allows us to endow parabolic structures at the generic points of the divisors.\end{proof} 

\subsubsection{The parahoric group scheme on the torus embedding ${{\bf Y}^{^{\text{aff}}}}$} 
\begin{thm} \label{mtY}  There exists an affine ``wonderful" Bruhat-Tits group scheme ${\mathcal G}_{_{{\bf Y}^{^{\text{aff}}}}}^{^{\varpi}}$ on $\bf Y$ together with a canonical isomorphism $\text{Lie}({\mathcal G}_{_{{\bf Y}^{^{\text{aff}}}}}^{^{\varpi}}) \simeq J$. It further satisfies the following classifying property:

For any point $h \in {{\bf Y}^{^{\text{aff}}}} \setminus T^{\ltimes}_{ad}$, let  $ \mathbb I \subset \mathbb S$ be a subset such that $h \in \cap_{\alpha \in \mathbb I} H_{_\alpha}$. Let $C_{_{\mathbb I}} \subset {{\bf Y}^{^{\text{aff}}}}$ be a smooth curve with generic point in $T^{\ltimes}_{_{\text{ad}}}$ and closed point $h$. Let $U_{_{h}} \subset C_{_{\mathbb I}}$ be a formal neighbourhood of the closed point $h$. Then, the restriction ${\mathcal G}_{_{{\bf Y}^{^{\text{aff}}}}}^{^{\varpi}}|_{_{U_{_{h}}}}$ is isomorphic to the Bruhat-Tits group scheme $\cG_{_{\mathbb I}}$ on $U_{_{h}}$.

\end{thm}

\begin{proof}
Recall that ${{\bf Y}^{^{\text{aff}}}}$ is covered by affine spaces $Y_{_w} \simeq \mathbb{A}^{\ell+1}$ parametrized by the affine Weyl group $W^{^{\text{aff}}}$. Each $Y_{_w}$ is the translate of ${{\bf Y}_{_0}^{^{\text{aff}}}}$.  The translates of the divisors $H_{_\alpha}$ meet each $Y_{_w}$ in the standard hyperplanes on $\mathbb{A}^{\ell+1}$ and thus we can prescribe the same ramification data at the hyperplanes on each of the $Y_{_w}$'s. On the other hand, although we have simple normal crossing singularities, we do not have an analogue of the Kawamata covering lemma for schemes such as ${{\bf Y}^{^{\text{aff}}}}$. The lemma is known only in the setting of quasi-projective schemes. So to construct the group scheme, we employ a different approach.

We observe firstly that, the formalism of Kawamata coverings applies in the setting of the affine spaces $Y_{_w} \subset {{\bf Y}^{^{\text{aff}}}}$. Let $p_w: Z_{_w} \ra Y_{_w}$ be the associated Kawamata cover (\ref{kawa}) with Galois group $\Gamma_w$. In Corollary \ref{canparstr} we observed that the Lie-algebra bundle $J$ has a canonical parabolic structure like $\cR$ in the proof  of Theorem \ref{btoverx}. Letting $Y_{_w}$ play the role of ${\bf Y_{_0}}$ and  using all arguments in the proof of  Theorem \ref{btoverx} we obtain $\cH_{_w} \ra Z_{_w}$ which is $\Gamma_w$-group scheme with  fibers isomorphic to   $G$ whose invariant direct image is a group scheme $\cG_w$ such that $\text{Lie}(\cG_w) =J|_{Y_{_w}}$. The induced parabolic structure on $J|_{Y_{_w}}$ is the restriction of the one on $J$.  Indeed, by \eqref{canparstr} these parabolic structures are essentially given at the local rings at the generic points $\xi_{\alpha}$ of the divisors $H_{_\alpha}$ and hence 
these parabolic structures on $J$ agree on the intersections $Y_{{_{uv}}}:=Y_{_u} \cap  Y_{_v}$. 

Let $Z_{_{uv}}:= p_u^{-1}(Y_{_{uv}})$. Let $\tilde{Z}_{_{uv}}$ be the normalization  of a component of  $Z_{_{uv}} \times_{Y_{_{uv}}} Z_{_{vu}}$. Then $\tilde{Z}_{_{uv}}$   serves as Kawamata cover (\ref{kawa}) of $Y_{_{uv}}$ (see \cite[Corollary 2.6, page 56]{vieweg}).  We consider the morphisms $\tilde{Z}_{_{uv}} \to Z_{_{u}}$ (resp.$\tilde{Z}_{_{uv}} \to Z_{_{v}}$) and let $\cH_{_{u,\tilde{Z}}}$ (resp. $\cH_{_{v,\tilde{Z}}}$) denote the pull-backs of $\cH_{_u}$ (resp. $\cH_{_v}$)  to $\tilde{Z}_{_{uv}}$. 

Let $\Gamma$ denote the Galois group for $\tilde{Z}_{_{uv}} \ra Y_{_{uv}}$. Then by Lemma \ref{Weilrestriction}, the invariant direct images of the equivariant Lie algebra bundles $\text{Lie}(\cH_{_{u,\tilde{Z}}})$ and $\text{Lie}(\cH_{_{v,\tilde{Z}}})$ coincide with the Lie algebra structure on   $J$ restricted to the ${Y_{_{uv}}}$ and also as isomorphic parabolic bundles. Therefore we have a natural isomorphism of equivariant Lie-algebra  bundles
\begin{equation}
\text{Lie}(\cH_{u,\tilde{Z}}) \simeq \text{Lie}(\cH_{v,\tilde{Z}}).
\end{equation}
As in the proof of Theorem \ref{btoverx}, this gives a canonical identification of the equivariant group schemes $\cH_{_{u,\tilde{Z}}}$ and $\cH_{_{v,\tilde{Z}}}$ on $\tilde{Z}_{_{uv}}$. Since the invariant direct image of both the group schemes  $\cH_{_{u,\tilde{Z}}}$ and $\cH_{_{v,\tilde{Z}}}$ are the restrictions 
$\cG_{_{u,{Y_{_{uv}}}}}$ and $\cG_{_{v,{Y_{_{uv}}}}}$, it follows that on $Y_{_{uv}}=Y_{_u} \cap Y_{_v}$  we get a canonical identification of group schemes:
\beqa\label{inducedfromlie}
\cG_{_{u,{Y_{_{uv}}}}} \simeq \cG_{_{v,{Y_{_{uv}}}}}.
\eeqa 
These identifications are canonically induced from the gluing data of the Lie algebra bundle $J$ for the cover $Y_{_w}$'s. Therefore the cocycle conditions are clearly satisfied and the identifications \eqref{inducedfromlie} glue to give  the group scheme ${\mathcal G}_{_{{\bf Y}^{^{\text{aff}}}}}^{^{\varpi}}$ 
 on ${\bf Y}^{^{\text{aff}}}$. The verification of the classifying property follows exactly as in the proof of Theorem \ref{btoverx}.

 \end{proof}

\section{The Bruhat-Tits group scheme on ${\bf X}^{^{{aff}}}$}
Let  ${\bf X}^{^{{aff}}}$ as in \S\ref{loopcase}. We begin with a generality. Let $\mathbb{X}$ be an ind-scheme. By an open-subscheme $i: \mathbb{U} \hookrightarrow \mathbb{X}$ we mean an ind-scheme such that for any $f: \spec(A) \rightarrow \mathbb{X}$, the natural morphism $\mathbb{U} \times_{\mathbb{X}} \spec(A) \rightarrow \spec(A)$ is an open immersion. For a sheaf $\mathbb{F}$ on $\mathbb{U}$ by $i_{_*}(\mathbb{F})$ we mean the sheaf associated to the pre-sheaf on the ``big site" of $\mathbb{X}$, whose sections on $f: \spec(A) \rightarrow \mathbb{X}$ are given by $\mathbb{F}(\mathbb{U} \times_{\mathbb{X}} \spec(A))$.

 The ind-scheme ${\bf X}^{^{\text{aff}}}$ has a certain open subset ${\bf X}_{_0}$ whose precise definition is somewhat technical \cite[Page 705]{solis}. Let us mention the properties relevant for us.

Recall that ${\bf X}^{^{\text{aff}}} = (G_{_{\text{ad}}}^{^{\text{aff}}} \times G_{_{\text{ad}}}^{^{\text{aff}}})~{\bf X}_{_0}$ and in fact ${\bf X}^{^{\text{aff}}} = (G_{_{\text{ad}}}^{^{\text{aff}}} \times G_{_{\text{ad}}}^{^{\text{aff}}}) ~{{\bf Y}^{^{\text{aff}}}}$. Further, the torus-embedding ${\bf Y}^{^{\text{aff}}}$ is covered by ${\bf Y}^{^{\text{aff}}}_{_w} \simeq \mathbb A^{^{\ell +1}}$ which are $W^{^{\text{aff}}}$-translates of ${\bf Y}^{^{\text{aff}}}_{_0}$, where ${\bf Y}^{^{\text{aff}}}_{_0} = {{\bf Y}^{^{\text{aff}}}} \cap {\bf X}_{_0}$. We remark that, analogous to the case of ${\bf Y}_{_0} \subset {\bf Y} \subset {\bf X}$, just as the toric variety ${\bf Y}_{_0} $ was associated to the negative Weyl chamber, the toric variety  ${\bf Y}^{^{\text{aff}}}_{_0}$ is associated to the negative Weyl alcove.

So let us denote by  $0$  the neutral element of $W^{^{\text{aff}}}$ also.  Let $U^{\pm} \subset B^{\pm}$ be the unipotent subgroups. Let $\cU^{\pm}:= ev^{-1}(U^{\pm})$ where $ev: G(A) \ra G(k)$ is the evalution map. Further by \cite[Prop 5.3]{solis} 
\begin{equation} {\bf X}_{_0}= \cU \times {\bf Y}^{^{\text{aff}}}_{_0} \times \cU^-.
\end{equation}
Let ${\bf X}_{_w} := \cU \times {\bf Y}^{^{\text{aff}}}_{_w} \times \cU^-$. These cover $\cU \times {{\bf Y}^{^{\text{aff}}}} \times \cU^-$. For  $g \in G_{_{\text{ad}}}^{^{\text{aff}}} \times G_{_{\text{ad}}}^{^{\text{aff}}}$, let 
\begin{equation} \label{compatibilities} {\bf X}_g:=g {\bf X}_0 \quad {\bf X}_{_{g,w}}:=g {\bf X}_{_w} \quad {\bf Y}^{^{\text{aff}}}_{_{g,w}}:=g {\bf Y}^{^{\text{aff}}}_{_w}.
\end{equation}
Note that we have the projection ${\bf X}_{_{g,w}} \to {\bf Y}^{^{\text{aff}}}_{_{g,w}}$ which is a $\cU \times \cU^-$-bundle.

 By \cite[Theorem 5.1]{solis} the ind-scheme ${\bf X}^{^{\text{aff}}}$ has divisors $D_\alpha$ for $\alpha \in \mathbb S$ such that the complement of their union is ${\bf X}^{^{\text{aff}}} \setminus G_{_{\text{ad}}}^{^{\text{aff}}}$. The next proposition shows the existence of a finite dimensional Lie algebra bundle on ${\bf X}^{^{\text{aff}}}$ which is analogous to the bundle $\cR$ over ${\bf X}$. 

\begin{prop}\label{isotliealginfty1} There is a finite dimensional Lie algebra bundle $\bf R$ on ${\bf X}^{^{\text{aff}}}$ which extends the trivial Lie algebra bundle $G_{_{\text{ad}}}^{^{\text{aff}}} \times \mathfrak g$ on the open dense subset $G_{_{\text{ad}}}^{^{\text{aff}}} \subset {\bf X}^{^{\text{aff}}}$ and whose restriction to ${\bf Y}^{^{\text{aff}}}$ is $J$. \end{prop}

\begin{proof} 
Following the arguments for  \eqref{gpschLiestab1} and \eqref{gpschLiestab2}, let $J'$ be the locally free sheaf obtained on the open subset ${\bf X'}$ obtained by the union of $G^{^{\text{aff}}}$ and the height one prime ideals of ${\bf X}^{^{\text{aff}}}$.  Let ${\bf R}$ denote its push-forward. To check that the push-forward is locally free, without loss of generality we may consider its restriction to ${\bf X}_0$. But on ${\bf X}_0$, the pushforward of $J'$ restricts to the pull-back of a Lie-algebra bundle $J$ on ${\bf Y}^{^{\text{aff}}}$ constructed in \eqref{liealgbunonyaff}, which completes the argument. The rest of the properties follows immediately.
\end{proof}

\begin{thm}  \label{gpshsolis} There exists an affine ``wonderful" Bruhat-Tits group scheme ${\mathcal G}_{_{{\bf X}^{^{\text{aff}}}}}^{^{\varpi}}$ on ${\bf X}^{^{\text{aff}}}$ together with a canonical isomorphism $\text{Lie}({\mathcal G}_{_{\bf X^{aff}}}^{^{\varpi}}) \simeq {\bf R}$. It further satisfies the following classifying property: 

For any point $h \in {\bf X}^{^{\text{aff}}} \setminus G_{_{\text{ad}}}^{^{\text{aff}}}$, let  $ \mathbb I \subset \mathbb S$ be defined by the condition $h \in \cap_{\alpha \in \mathbb{I}} D_\alpha$. Let $C_{_{\mathbb I}} \subset {\bf X}$ be a smooth curve with generic point in $G_{_{\text{ad}}}^{^{\text{aff}}}$ with closed point $h$. Let $U_{_{h}} \subset C_{_{\mathbb I}}$ be a formal  neighbourhood of $h$. Then, the restriction ${\mathcal G}_{_{{\bf X}^{^{\text{aff}}}}}^{^{\varpi}}|_{_{U_{_{h}}}}$ is isomorphic to the Bruhat-Tits group scheme $\cG_{_{\mathbb I}}$  on $\spec(A)$. 
\end{thm}
\begin{proof} We begin by observing  that ${\bf X}_{_{g,w}}$ and ${\bf R} \ra {\bf X}^{^{\text{aff}}}$ play the role of ${{\bf Y}_{_{g,w}}^{^{\text{aff}}}}$ and $J \ra {{\bf Y}^{^{\text{aff}}}}$ in the proof of Theorem \ref{mtY}.  Therefore the group scheme $\cG_{_{g,w}}$ glue together and we obtain the global group scheme ${\mathcal G}_{_{{\bf X}^{^{\text{aff}}}}}^{^{\varpi}}$. The verification of the classifying property follows exactly as in the proof of Theorem \ref{btoverx}.

\end{proof}

\section{The Bruhat-Tits group scheme on ${\bf X}^{^{{poly}}}$}

In \cite{solis}, apart from the loop group case Solis also discusses the polynomial loop group situation which in fact is relatively simpler. Consequently, the analysis as well as the results in the loop group situation go through with some simplifications. Let  $L_{_{\text{poly}}}G$ be defined by 
\begin{equation} L_{_{\text{poly}}}G(R') = G(R'[z^{^{\pm}}])
\end{equation} and we define $L_{_{\text{poly}}}^{\ltimes}G$  similarly as $L_{_{\text{poly}}}^{\ltimes}G 
:= \GG_m \ltimes L_{_{\text{poly}}}G$. Solis constructs a ``wonderful" embedding of the polynomial loop group $L_{_{\text{poly}}}^{\ltimes}G/Z(G)$. We denote this space by ${\bf X}^{^{\text{poly}}}$.  It has an open subset ${\bf X}^{^{\text{poly}}}_{_0}$ whose definition is somewhat technical \cite[\S 5.3]{solis}, but, which plays the role completely analogus to ${\bf X}_{_0}$ for ${\bf X}^{^{\text{aff}}}$.  We will continue to denote by $\bf Y:= \ol{T^{\ltimes}_{_{\text{ad}}}}$ the toral embedding of $T^{\ltimes}_{_{\text{ad}}}$ obtained by taking its closure  in ${\bf X}^{^{\text{poly}}}$. As in the loop group case, in the polynomial case too $\bf Y$ is covered by infinitely  many affine spaces  $\mathbb{A}^{\ell+1}$ parametrized by the affine Weyl group $W^{^{\text{aff}}}$. In fact, if $t^{-\alpha_i}$ are the regular functions on  $T^{\ltimes}_{_{\text{ad}}}$ given by the character $-\alpha_i$, then 
$$\ol{T^{\ltimes}_{_{ad,0}}} := \ol{T^{\ltimes}_{_{\text{ad}}}} \cap {\bf X}_{_0}^{^{poly}} \simeq  \spec \CC[t^{-\alpha_0}, \ldots, t^{-\alpha_\ell}].$$

We index these affine spaces by the affine Weyl group $W^{^{\text{aff}}}$ and denote them as before by $Y_{_w}$'s, $w \in W^{^{\text{aff}}}$. As before, we may construct a finite-dimensional Lie-algebra bundle $J$ on ${{\bf Y}^{^{\text{aff}}}}$, then ${\bf R}$ on ${\bf X}^{^{\text{poly}}}$ and then construct a group scheme ${\mathcal G}_{_{{\bf X}^{^{\text{poly}}}}}^{^{\varpi}}$. Since the inductive structure of ${\bf X}^{^{\text{poly}}}$ is identical to ${\bf X}^{^{\text{aff}}}$,  the proofs are also identical. Therefore we omit them.

Now we outline a construction in the polynomial loop situation which will be generalized in the next section. For any affine space $Y_{_w}$, we have an affine embeddings
\beqa
i_w: T_{_{\text{ad},w}} \times \GG_m = T^{\ltimes}_{_{\text{ad},w}} \hra Y_{_w}
\eeqa
Furthermore, the projection $p:T_{_{\text{ad},w}} \times \GG_m \to \GG_m$ extends to a morphism $p_w:Y_{_w} \to \mathbb A^1$ such that $p^{^{-1}}(0)$ is the union of the hyperplanes. At the level of coordinates, this extension is simply the product of the coordinate functions. Further $i_w$ and $p_w$ glue to give
\beqa\label{keydiagpolyetale}
\xymatrix{
T_{_{\text{ad}}}^{\ltimes} \ar[r] \ar[d]_{p} &
 {{\bf Y}^{^{\text{aff}}}} \ar[d]^{p} \\
\mathbb{G}_m \ar[r] &  \mathbb{A}^1
}
\eeqa

\section{An analogue of a construction of Mumford} In 
\cite{kkms} towards the very end Mumford gives a beautiful construction of the geometric realization of the relative case of buildings via toroidal embeddings. He deals with the general situation of an arbitrary discrete valuation ring $R$ with algebraically closed residue field $k = R/\mathfrak m$. Our aim in this section is limited to the case when $k = \CC$, $\mathfrak{m}=(z) \subset \CC[z]$ and $R = \CC[z]_{_{\mathfrak m}}$.  Let $K = \CC(z)$ and let $S = \spec R$ and $\eta$ be the generic point and $o \in S$ the closed point. For schemes $X$ over $S$, $X_{_\eta}$ will be the fibre over $\eta$ and $X_{_o}$ the fibre over $o$.

We will briefly sketch an analogue of Mumford's construction in our setting and outline of the relationship between this construction to Solis's approach and as a consequence construct a ``wonderful" Bruhat-Tits group scheme on the toroidal embedding of $G_{_{ad,}} \times S$.

{\em For the purposes of this section alone so as to remain consistent with the one in \cite{kkms}, we will have the following set of notations}.
\beqa\label{newnots}
H := G_{_{\text{ad}}}\\
H_{_S} := G \times S\\
T_{_S} \subset H_{_S}\\
T_{_S} := T_{_{\text{ad}}} \times S\\
T_{_K} := T_{_{\text{ad}}} \times \spec K
\eeqa
i.e. $H_{_S}$ is a split {\em adjoint} semisimple group scheme over $S$, $T_{_S} \subset H_{_S}$ a fixed split maximal torus and $G$ will as before stand for the simply connected group. We will denote $U^{\pm}_S:=U^{\pm} \times S$.

Base-changing (\ref{keydiagpolyetale}) by $\spec R \rightarrow \mathbb{A}^1$, by the definition of $\bf Y$ we obtain:
\beqa\label{keydiagetale}
\xymatrix{
T_{_{K}} \ar[r] \ar[d]_{p} &
 {{\bf Y}^{^{\text{aff}}}} \ar[d]^{p} \\
\spec K \ar[r] &  \spec R
}
\eeqa
We continue to denote by $Y_{_w}$ the base-change.
The orbit space $H_{_S}/T_{_S}$ exists as an affine scheme over $S$ and $H_{_S}$ via the quotient map $\pi:H_{_S} \to H_{_S}/T_{_S}$ is a locally free principal $T_{_S}$-bundle over $H_{_S}/T_{_S}$.
We now consider the associated fibre bundle:
\beqa\label{normalplusmodif}
\xymatrix{
H_{_S} \ar[dr]_{\text{fibre}=T} \ar@{^{(}->}[rr]& &H_{_S}\times^{^{T_{_S}}} Y_{_w} \ar[dl]^{\text{fibre}=Y_{_w}} \\
& H_{_S}/T_{_S} &
}
\eeqa
Note that $T_\eta = Y_{w,\eta}$ and hence $H_\eta = 
(H_S \times^{^{T_S}} Y_{_w})_{_{\eta}}$. Let us denote by 
\beqa
\ol{Z}_w := H_S \times^{^{T_S}} Y_{_w}
\eeqa
and observe that $U_{_S}^- \times Y_{_w} \times U_{_S} = Z_{_w}$ is an open subset of $\ol{Z}_w$. In fact, over $U_{_S}^- \times U_{_S} \subset H_{_S}/T_{_S}$,  the quotient map $\pi:H_S \ra H_S/T_S$ is trivial.

Analogous to Mumford's definition \cite[page 206]{kkms}, we define
\beqa
\ol{H_{_S}} := \bigcup_{_{x \in H(K)}} (H_S \times^{^{T_S}} Y_{_w}).x
\eeqa
where the generic fibre $H_\eta$ is identified in each copy $(H_S \times^{^{T_S}} Y_{_w}).x$ by a translate of $H_\eta$ by right multiplication by $x$ in the Iwahori subgroup. It takes a bit to prove that there is such an action. Note that we have the identification:
\beqa
H_{\eta} = \ol{H_{_{\eta}}}
\eeqa

Consider the inclusion $G_{_{\text{ad}}} \times \GG_m \hra L_{_{\text{poly}}}^{\ltimes}G/Z(G)$. Let ${\bf M} = \ol{G_{_{\text{ad}}} \times \GG_m}$ in ${\bf X}^{^{poly}}$ as a scheme over $\CC$. It is {\em locally of finite type.}

\begin{prop} The $S$-scheme $\ol{H_{_S}}$ as a $\CC$-scheme is isomorphic to $\bf M$.
\end{prop}
In particular, $\ol{H_{_S}}$ is a regular scheme over $\CC$ such that the closed fibre $\ol{H_{o}}$ is a complete scheme over $\CC$ which is a union of smooth components meeting at normal crossings. The embedding $H_{_S} \subset \ol{H_{_S}}$ is a toroidal embedding without self-intersection. Finally, the strata of $\ol{H_{_S}} - H_{_S}$ are precisely the parahoric subgroups of $H(K)$. This bijection between the strata and the parahoric subgroups extends to an isomorphism of the graph of the embedding $H_{_S} \subset \ol{H_{_S}}$ with the Bruhat-Tits building of $H$ over $S$. 

By the methods in this note we immediately deduce the existence of a ``wonderful" Bruhat-Tits group scheme $\mathcal G_{_{\bf M}}^{^{\varpi}}$ on $\bf M$ which behaves naturally with respect to the strata.
\begin{rem} A difference between Mumford's construction and ours is that in Mumford's construction the fibre over the closed point in $S$ is a reducible divisor but not necessarily with simple normal crossing singularities. In fact, the components associated to the parahorics which are not hyperspecial come with multiplicity being the coefficient $c_{_\alpha}$ of the associated simple root in the expression of the highest root. So a Kawamata lemma is not immediately applicable in that situation.
\end{rem}
\begin{rem}\label{prelative} Mumford's construction works over algebraically closed fields of any characteristic while we assumed the characteristic to be $0$. We did this to make the construction in the loop group situation following \cite{solis}. On the other hand, it seems  that with a bit more work the whole construction of $\mathcal G_{_{\bf M}}^{^{\varpi}}$ will go through for characteristics $p$ such that $p$ is coprime to the $d_{_\alpha}$'s in \eqref{dalpha}. \end{rem}

\section{Appendix on parabolic and equivariant bundles}\label{parabstuff} In this section we recall and summarize results on parabolic bundles and equivariant bundles on Kawamata covers. These play a central role in the constructions of the Bruhat-Tits group schemes made above. Consider a pair $(X,D)$, where $X$ is a smooth quasi-projective variety and $D = \sum_{_{j = 0}}^{^{\ell}} D_{_j}$ is a {\em reduced normal crossing divisor}  with non-singular components $D_{_j}$  intersecting each other transversely. The basic  examples we have in mind in the note are discrete valuation rings with its closed point, the wonderful compactification $\bf X$ with its boundary divisors or affine toric varieties.

Let $E$ be a locally free sheaf on $X$.  Let $n_{_j}, j = 0 \ldots, \ell$ be positive integers attached to the components $D_{_j}$. Let $\xi$ be a generic point of $D$. 

Let $E_{_\xi} := E \otimes_{_{\mathcal O_{_{X}}}} \mathcal O_{_{X,\xi}}$ and $\bar{E}_{_{\xi}} := E_{_{\xi}}/\mathfrak m_{_{\xi}} E_{_{\xi}}$, $\mathfrak m_{_{\xi}}$ the maximal ideal of $\mathcal O_{_{X,\xi}}$. 
\begin{defi}\label{sestype} A parabolic structure on $E$ consists of the following data:
\begin{itemize}
\item A flag $\bar{E}_{_{\xi}} = F^{^1}\bar{E}_{_{\xi}} \supset \ldots \supset F^{r{_{_j}}}\bar{E}_{_{\xi}}$, at the generic point $\xi$ of each of the components $D_{_j}$ of $D$,
\item Weights $d_{_s}/n_{_j}$, with $0 \leq d_{_s} < n_{_j}$ attached to $F^{^s}\bar{E}_{_{\xi}} $  such  that  $d_1 < \cdots < d_{r_j}$
\end{itemize}
\end{defi}
By saturating the flag datum on each of the divisors we get 
for each component $D_{_j}$ a filtration:
\beqa
E_{_{D_{_j}}} = F^{^1}_{_{j}} \supset \ldots \supset F^{r{_{_j}}}_{_{j}}
\eeqa
of sub-sheaves on $D_{_j}$. Define the coherent subsheaf $\mathcal F^{s}_{_j}$, where $0 \leq j \leq \ell$ and $1 \leq s \leq r_{_j}$, of $E$ by:
\beqa
0 \to \mathcal F^{s}_{_j} \to E \to E_{_{D_{_j}}}/F^{s}_{_j} \to 0
\eeqa
the last map is by restriction to the divisors. In \cite{biswas} it is shown that there is a Kawamata covering of $p:(Y,\tilde{D}) \to (X,D)$  with suitable ramification data and Galois group $\Gamma$. It is also shown that there is a $\Gamma$-equivariant vector bundle $V$ on $Y$ such that the invariant direct image $p^{\Gamma}_{_*}(V) = E$ which also recovers the  filtrations $\mathcal F^{s}_{_j}$.

For the sake of completeness, we give a self-contained {\em ad hoc} argument for this construction which is more in the spirit of the present note. The data given in \eqref{sestype} is as in \cite{ms}   which deals with points on curves. {\em Note that in our setting we have rational weights}. Under these conditions, the aim is to set up an equivalence 
\beqa\label{bisses}
p^{\Gamma}_{_*}: \{\Gamma-\text{equivariant~bundles~on}~Y\} \to \{\text{parabolic~bundles~on}~X\}
\eeqa
and as the notation suggests, this is achieved by taking invariant direct images. Since $p:Y \to X$ is finite and flat, if $V$ is locally free on $Y$ then so is $p^{\Gamma}_{_*}(V)$. An equivariant bundle $V$ on $Y$ is defined in an analytic neighbourhood of a generic point $\zeta$ of a component by a representation of the isotropy group $\Gamma_{_{\zeta}}$ and locally (in the analytic sense) the action of $\Gamma_{_{\zeta}}$ is the product action. This is called the local type in \cite{pibundles} or \cite{ms}. By appealing to the $1$-dimensional case, we get canonical parabolic structures on $E$ at the generic points $\xi$ of the components $D_{_j}$ in $X$. 

This construction is an equivalence. Given a parabolic bundle $E$ on $(X,D)$, to get the $\Gamma$-equivariant bundle $V$ on $Y$, such that $p^{\Gamma}_{_*}(V) = E$, we proceed as follows. If we did know the existence of such a $V$ then we can consider the inclusion $p^*(E) \subset V$. Taking its dual (and since $V/p^*(E)$ is torsion, taking duals is an inclusion):
\beqa
V^* \hra  (p^*(E))^* = p^*(E^*)
\eeqa
By the $1$-dimensional case (obtained by restricting to the height $1$-primes at the generic points) we see that the quotient $T_{_\zeta} := p^*(E^*)_{_\zeta}/V_{_\zeta}^*$ is a torsion $\mathcal O_{_{Y,\zeta}}$-module and $T_{_\zeta}$ is completely determined by the parabolic structure on $E$. 

The parabolic structure on $E$ therefore determines canonical quotients:
\beqa\label{quotients}
p^*(E^*)_{_\zeta} \to T_{_\zeta}
\eeqa
for each generic point $\zeta$ of components in $Y$ above the components $D_{_j}$ in $X$.

The discussion above suggest how one would construct such a $V$;  we begin with these quotients \eqref{quotients}. Then we observe that there is a maximal coherent subsheaf $V' \subset p^*(E^*)$ such that 
\beqa
p^*(E^*)_{_\zeta}/V_{_\zeta}' = T_{_\zeta}.
\eeqa
Away from $p^{^{-1}}(D)$ in $Y$, the inclusion $V' \hra p^*(E^*)$ is an isomorphism. Dualizing again, we get an inclusion $p^*(E) \hra (V')^{^*}$ which is an isomorphism away from $p^{^{-1}}(D)$. Set $W =  (V')^{^*}$. Since $W$ is the dual of a coherent sheaf, it is reflexive. It is not hard to check that $W$ coincides with the vector bundle $V$ that Biswas constructs in codimension $2$ and hence everywhere by \cite[Proposition 1.6, page 126]{hartshorne}. This gives us the desired equivalence \eqref{bisses}.
\subsubsection{The group scheme situation}\label{grstuff} Let $(X,D)$ be as above. Let $\xi$ be the generic point of a component $D_{_j}$ of $D$ and let $A := \mathcal O_{_{X,\xi}}$ and $K$ be the quotient field of $A$. Let $\mathcal G_{_{\theta}}$ be a Bruhat-Tits group scheme on $\spec(A)$ associated to a vertex $\theta$ of the Weyl alcove. {\em We always assume that these group scheme are generically split}.

Let $B = \mathcal O_{_{Y, \zeta}}$ where $\zeta$ is the generic point of a component of $Y$ above $D_{_j}$ and let $L$ be the quotient field of $B$. We assume that the local ramification data for the Kawamata covering has  numbers $d_{_\alpha}$ \eqref{dalpha}.  Let $\Gamma_{_\zeta}$ be the stabilizer of $\Gamma$ at $\zeta \in Y$.

The results of \cite[Proposition 5.1.2 and Remark 2.3.3]{base} show that there exists an equivariant group scheme $\mathcal H_{_B}$ on $\spec(B)$ with fiber isomorphic to the simple connected group $G$ and such that $\text{Res}_{_{B/A}}(\mathcal H_{_B})^{^{\Gamma_{_{\zeta}}}} \simeq \mathcal G_{_{\theta}}$.

Suppose that we have a group scheme $\mathcal G_{_{X'}}$ on an open $X' \subset X$ which includes all the height $1$ primes coming from the divisors $D_{_j}$'s with the following properties:
\begin{itemize}
\item Away from the divisor $D \subset X$, $\mathcal G$ is the constant group scheme with fiber $G$.
\item The restrictions $\mathcal G\mid_{_{\spec(A)}}$ at the generic points $\xi$ are isomorphic to the Bruhat-Tits  group scheme $\mathcal G_{_\theta}$ for varying $\xi$ and $\theta$.
\end{itemize}  
In other words, $\mathcal G_{_{X'}}$ is obtained by a gluing of the constant group schemes with $\mathcal G_{_\theta}$'s along $\spec(K)$ by an automorphism of constant group scheme $G_{_K}$.

Now consider the inverse image of the constant group scheme $p^{^*}(G_{_{X - D}}) \simeq G \times p^*(X-D)$. Then using the gluing on $X'$, we can glue the constant group scheme $p^{^*}(G_{_{X - D}})$ with the local group schemes $\mathcal H_{_B}$ for each generic point $\zeta$ to obtain a group scheme $\mathcal H_{_{Y'}}$ on $Y' = p^{^{-1}}(X')$ such that:
\beqa
\text{Res}_{_{Y'/X'}}(\mathcal H_{_{Y'}})^{^{\Gamma}} \simeq \mathcal G_{_{X'}}.
\eeqa
\subsubsection {Kawamata Coverings}\label{kawa}
Let $X$ be a smooth quasi-projective variety and let $D = \sum_{_{i=0}}^{^\ell} D_i$
be the decomposition of the simple or reduced normal crossing divisor
$D$ into its smooth components (intersecting transversally).
The ``Covering Lemma'' of Y. Kawamata
(see \cite[Lemma 2.5, page 56]{vieweg}) says
that, given positive integers $n_{_0}, \ldots, n_{_\ell}$, there is a connected smooth quasi-projective
variety $Z$ over $\mathbb C$ and a Galois covering morphism
\beqa\label{kawamatacm}
\kappa:Z \to  X 
\eeqa
such that the reduced divisor $\kappa^{^*}{D}:= \,({\kappa}^{^*}D)_{_{\text{red}}}$
is a normal crossing divisor on $Z$ and furthermore,
${\kappa}^{^*}D_{_i}= n_{_i}.({\kappa}^{^*}D_{i})_{_{\text{red}}}$. Let $\Gamma$ denote the Galois group
for the covering map $\kappa$.

The isotropy group of any point $z \in Z$, for the
action of $\Gamma$ on $Z$, will be denoted by ${\Gamma}_{_z}$.  It is easy to see that the stabilizer at generic points of the irreducible components of $(\kappa^{^*}D_i)_{_{\text{red}}}$ are cyclic of order $n_{_i}$.


\begin{thebibliography}{9999}
\bibitem{base} V.Balaji and C.S. Seshadri, Moduli of parahoric $\mathcal G$--torsors on a compact Riemann surface, {\em J. Algebraic Geometry}, 24, (2015),1-49.
\bibitem{biswas} I. Biswas, Chern classes for parabolic bundles. J. Math. Kyoto Univ. 37 (1997), no. 4, 597--613. doi:10.1215/kjm/1250518206.
\bibitem{blr} S. Bosch, W.Lutkebohmert and M.Raynaud, { Neron
  Models}, Ergebnisse 21, Springer Verlag, (1990).
\bibitem{brion1} M. Brion, The behaviour at infinity of the Bruhat decomposition, Comm. Math. Helv. 73 (1998)137-174.  
\bibitem{brion} M. Brion, Log homogeneous varieties, Actas del XVI Coloquio Latinoamericano de {\'A}lgebra, 1–39, Biblioteca de la Revista Matematica Iberoamericana, Madrid, 2007.
\bibitem{bruhattits} F.Bruhat and J.Tits, Groupes r\'eductifs sur un crops local II: Sch\'emas en groupes. Existence d'une donn\'ee radicielle valu\'ee, {Publications Math\'ematiques de l'IH\'ES} 60 (1984) 5-184.
\bibitem{decp} C. De Concini and C. Procesi, Complete symmetric varieties, in Invariant theory (Montecatini, 1982), pp. 1–44, Lecture Notes in Math. 996, Springer, 1983.
\bibitem{DS} C. De Concini and T. Springer, Compactification of Symmetric
Varieties, Transform. Groups 4 (1999), no. 2-3, pp. 273-300.
\bibitem{cgp} B. Conrad, O.Gabber and Gopal Prasad, Pseudo-reductive groups, {\it Cambridge University Press}, New Mathematical Monographs, No 17.(2010), 2nd Edition.
\bibitem{edix} B. {Edixhoven}, Neron models and tame ramification, {\it Compositio Mathematica}, {\bf 81}, (1992), 291-306.
\bibitem{ega} A. Grothendieck, \'Elements de \'Geom\'etrie Alg\'ebrique, Math. Publ. IHES 4, (1960), 5-228.
\bibitem{hartshorne} R. Hartshorne, Stable Reflexive Sheaves, Mathematische Annalen 254 (1980), 121-176.

\bibitem{kkms} G. Kempf, F. Knudsen, D. Mumford, and B. Saint-Donat, Toroidal Embeddings I, Springer Lecture Notes 339, 1973.
\bibitem{ms} V. { Mehta and C.S. Seshadri}, Moduli of vector bundles on curves with parabolic structures, 
{\it Math. Ann}, {\bf 248}, (1980), 205-239. 

\bibitem{pibundles} C.S. { Seshadri}, Moduli of $\pi$--vector bundles over an algebraic curve, Questions On
algebraic Varieties, C.I.M.E, Varenna, (1969), 139-261. 
\bibitem{solis} P. Solis, A wonderful embedding of the loop group. {\em Advances in Mathematics} 313, , (2017), 689-717.
\bibitem{El} E. Strickland, A vanishing Theorem for Group Compactifications, Math. Ann. 277 (1987), pp. 165-171.
\bibitem{vieweg} E. Vieweg, Projective Moduli for Polarized Manifolds, , Ergebnisse, Springer(1995).
\end{thebibliography}
\end{document}